\documentclass[a4paper,twoside,11pt]{elsarticle}

\usepackage{color}

\definecolor{dgreen}{rgb}{0,0.6,0.1}

\definecolor{dred}{rgb}{0.8,0,0}

\definecolor{dmagenta}{rgb}{0.6,0,0.6}

\definecolor{dmag}{rgb}{0.8,0,0.9}

\definecolor{dblue}{rgb}{0,0,0.7}

\definecolor{dbrown}{rgb}{0.8,0.25,0.25}

\definecolor{doran}{rgb}{1,0.4,0}

\usepackage{todonotes}

\usepackage[percent]{overpic}
\usepackage{relsize}
\usepackage{multirow}
\usepackage{siunitx}

\usepackage{amsmath,amssymb,amsthm}
\usepackage{graphicx}
\usepackage{multirow}

\newtheorem{thm}{Theorem}[section]
\newtheorem{prop}[thm]{Proposition}

\newtheorem{remark}{Remark}

\newcommand{\vshortarrow}[1]{\xrightarrow[\hspace*{0.03cm}]{\raisebox{1mm}{\normalsize$#1$}}}

\newcommand\Om{{\Omega}}
\newcommand\E{{\PP}}

\newcommand\bgrad{\textrm{\bf grad}}
\renewcommand\div{{\rm div}}
\newcommand\bcurl{\textrm{\bf curl}}
\newcommand\curl{\textrm{\bf curl}}
\newcommand\brot{\textrm{\bf rot}}
\newcommand\rot{{\rm rot}}

\newcommand{\dO}{\,{\rm d}\Omega}
\newcommand{\dE}{\,{\rm d}\E}
\newcommand{\dPP}{\,{\rm d}\PP}
\newcommand{\ds}{\,{\rm d}s}
\newcommand{\dS}{\,{\rm d}S}

\newcommand{\df}{\,{\rm d}f}

\newcommand\nn{\boldsymbol n}
\renewcommand\tt{\boldsymbol t}

\newcommand\xxP{{\boldsymbol x}_{\PP}}
\newcommand\xxf{{\boldsymbol x}_{f}}

\newcommand{\X}{\textnormal{$\mathbf{x}$}}

\newcommand\qq{\boldsymbol q}

\newcommand\bphi{\boldsymbol \varphi}

\newcommand\vv{\boldsymbol v}

\renewcommand\qq{\boldsymbol q}
\newcommand\pp{\boldsymbol p}

\newcommand\HH{\boldsymbol H}
\newcommand\BB{\boldsymbol B}
\newcommand\jj{\boldsymbol j}

\newcommand\xx{\boldsymbol x}

\newcommand\ww{\boldsymbol w}

\renewcommand{\O}{ {\mathcal O}}

\newcommand\Th{{\mathcal T}_h}

\newcommand\R{\mathbb{R}}
\renewcommand{\P}{ {\mathbb P}}

\newcommand\PP{{\text P}}

 \numberwithin{equation}{section}
 
\newcommand{\node}{\rm{n}}
\newcommand{\edge}{\rm{e}}
\newcommand{\face}{\rm{f}}
\newcommand{\vol}{\rm{v}}
\newcommand\fpsP{V^{\face}_{k-1}(\PP)}
\newcommand\epsf{V^{\edge}_{k}(f)}

\newcommand\epsmP{{\edge,\mu,\PP}}

\definecolor{dgreen}{rgb}{0,0.6,0.1}

\definecolor{dred}{rgb}{0.8,0,0}

\definecolor{dmagenta}{rgb}{0.6,0,0.6}

\definecolor{dmag}{rgb}{0.8,0,0.9}

\definecolor{dblue}{rgb}{0,0,0.7}

\definecolor{dbrown}{rgb}{0.8,0.25,0.25}

\definecolor{doran}{rgb}{1,0.4,0}

\begin{document}
\begin{frontmatter}
\title{A family of three-dimensional virtual elements with applications to magnetostatic.}

\author[unimib,imati]{L. Beir\~ao da Veiga}
\ead{lourenco.beirao@unimib.it}

\author[imati]{F. Brezzi\corref{correspondingauthor}}
\ead{brezzi@imati.cnr.it}

\author[unimib]{F. Dassi}
\ead{franco.dassi@unimib.it}

\author[unipv,imati]{L.D. Marini}
\ead{marini@imati.cnr.it}

\author[unimib,imati]{A. Russo}
\ead{alessandro.russo@unimib.it}

\address[unimib]{Dipartimento di Matematica e Applicazioni, Universit\`a di 
Milano--Bicocca, Via Cozzi 53, I-20153, Milano, Italy}
\address[imati]{IMATI CNR, Via Ferrata 1, I-27100 Pavia, Italy}
\address[unipv]{Dipartimento di Matematica, Universit\`a di Pavia,
Via Ferrata 5, I-27100 Pavia, Italy}
\cortext[correspondingauthor]{Corresponding author}

\begin{abstract}
We consider, as a simple model problem, the application of Virtual Element Methods (VEM) to the linear Magnetostatic three-dimensional problem in the formulation of F. Kikuchi. In doing so, we also introduce new serendipity VEM spaces, where the serendipity reduction is made only on the faces of a general polyhedral decomposition (assuming that internal degrees of freedom could be more easily eliminated by static condensation). These new spaces are meant, more generally, for the combined approximation of $H^1$-conforming ($0$-forms), $H({\rm {\bf curl}})$-conforming ($1$-forms),  and $H({\rm div})$-conforming ($2$-forms) functional spaces in three dimensions, and they would surely be useful for other problems and in more general contexts.
\end{abstract}


\begin{keyword}
 Virtual Element Methods, $~$ Serendipity, $~$ Magnetostatic problems,

 AMS Subject Classification: 65N30
\end{keyword}

\end{frontmatter}

\section{Introduction}
The aim of this paper is two-fold. We present a variant of the serendipity nodal, edge, and face Virtual Elements presented in \cite{SERE-mix} that could be used in many different applications (in particular since they can be set  in an exact sequence), and we show their use on a model linear Magnetostatic problem in three dimensions, following the formulation of F. Kikuchi \cite{KikuchiJJAM}, \cite{KikuchiIEEE}. Even though such formulation is not widely used within the Electromagnetic computational community, we believe that is it a very nice example of use of the De Rham diagram (see e.g. \cite{Demko-hp}) that here is available for serendipity spaces of general order.

Virtual Elements were  introduced  a few years ago \cite{volley, hitchhikers, super-misti}, and can be seen as part of  the wider family of Galerkin approximations based on polytopal decompositions, including Mimetic Finite Difference methods (the {\it ancestors} of VEM: see e.g. \cite{MFD22,  MFD23} and the references therein),  Discontinuous Galerkin  (see e.g. \cite{ABCM,  Cockburn-ZAMM}, or recently  \cite{Dole-Feista}, and the references therein), Hybridizable Discontinuous Galerkin and their variants (see \cite{Cockburn-Jay-Lazarov},
or much more recently \cite{Cockburn-DiPietro-Ern, skeletal-es}, and the references therein). On the other hand their use of
non-polynomial basis functions connect them as well with other methods such as polygonal interpolant basis functions, barycentric coordinates, mean value coordinates, metric coordinate method, natural neighbor-based coordinates, generalized FEMs, and maximum entropy shape functions. See for instance
\cite{Wachspress11},
\cite{FloaterActa},
\cite{Sukumar:Malsch:2006},
\cite{TPPM10} and the references therein.
Finally, many aspects are closely connected with
Finite Volumes and related methods (see e.g. ~\cite{Droniou-gradient},
\cite{Droniou:Eymard:Gallouet:Herbin:2010},
and the references therein).

The list of VEM contributions in the literature is nowadays quite large; in addition to the ones above, we here limit ourselves to mentioning \cite{Stokes:divfree, HR-Virtual-Plate, VEM-elasticity, Berrone-VEM, brezzi:Marini:plates, Paulino-VEM, Gatica-1, Helmo-PPR, Wriggers-1}.

Here we deal, as a simple model problem, with the classical magnetostatic problem in a { smooth-enough} bounded  domain
$\Om$ in $\R^3$, simply connected: given $\jj \in H(\div;\Om)$ with $\div \jj =0$  in $\Om$, and given  $\mu=\mu(\xx) \mbox{ with }0<\mu_0\le \mu \le \mu_1$,
\begin{equation}
\left\{
\begin{aligned}\label{Max3}
& \mbox {find  }\HH\in H(\bcurl;\Om) \mbox{ and }\BB\in H(\div;\Om) \mbox{ such that: }\\
& \bcurl \HH=\jj  \mbox{ and }\div\BB=0, \mbox{ with }\BB=\mu\HH \mbox{ in }\Om,\\
&\mbox{with the boundary conditions } \HH\wedge\nn =0\mbox{ on }\partial\Om .
\end{aligned}
\right.
\end{equation}

When discretizing a three-dimensional problem, the degrees of freedom internal to elements (tetrahedra, hexahedra, polyhedra, etc.) can,  in most cases, be easily eliminated by {\it static condendation}, and their burden on the resolution of the final linear system is not overwhelming. This is not the case for edges and faces, where static condensation would definitely be much more problematic. On edges one cannot save too much:  in general the trial and test functions, there, are just one-dimensional polynomials. On faces, however, 
for $0$-forms and $1$-forms,  higher order approximations on polygons with many edges find a substantial benefit by the use of the serendipity approach, that allows an important saving of degrees of freedom internal to faces.

For that we constructed serendipity virtual elements in \cite{SERE-nod} and \cite{SERE-mix} (for scalar or vector valued
local spaces, respectively) that however were not fully adapted to the construction of De Rahm complexes. The spaces were therefore modified, for the 2d case, in \cite{max2}. Here we use this latest version on the {\it boundary} of the polyhedra of our three-dimensional decompositions, and we show that this can be a quite viable choice.

We point out that, contrary to what happens for FEMs (where, typically, the serendipity subspaces do not depend on the degrees of freedom used in the bigger, non-serendipity, spaces), for Virtual Elements the construction of the serendipity spaces depends, in general, heavily on the degrees of freedom used,  so that if we want an exact sequence the {\it degrees of freedom} in the VEM spaces must be chosen properly.

We will show that the present serendipity VEM spaces are perfectly suited for the approximation of problem \eqref{Max3} with the Kikuchi approach, and we believe that they might be quite interesting in many other problems in Electromagnetism as well as in other important applications of Scientific Computing.  In particular we have {\it a whole family} of spaces of different order of accuracy $k$. For simplicity we assumed here  that the same order $k$ is used in all the elements of the decomposition, but we point out that the great versatility of VEM would very easily comply with the use of different orders in different elements, allowing very effective $h/p$ strategies.

A single (lowest order only, and particularly cheap) Virtual Element Method for electro-magnetic problems was already  proposed in \cite{lowest-max3}, but the family proposed here does not include it: roughly speaking, the element in \cite{lowest-max3} is based on a generalization to polyhedra of the {\it lowest order N\'ed\'elec first type} element (say, of degree between $0$ and $1$), while, instead, the family presented here could be seen as being based on generalizations to polyhedra of the {\it  N\'ed\'elec second type} elements (of order $k\ge 1$).

A layout of the paper is as follows: in Section \ref{propol} we introduce some basic notation, and recall some well known properties of polynomial spaces. In Section \ref{discre1} we will first recall the Kikuchi variational formulation of \eqref{Max3}. Then, in Subsection \ref{local} we present the {\it local} two-dimensional Virtual Element spaces of {\it nodal} and {\it  edge}  type to be used on the interelement boundaries. As we mentioned already, {\it the spaces} are the same already discussed in \cite{volley}, \cite{projectors} and in \cite{BFM}, \cite{super-misti}, respectively,
but with a different choice of the  {\it degrees of freedom}, suitable for the serendipity construction discussed in Subsection \ref{sere}. In Subsection \ref{local3} we  present the {\it local}  three-dimensional spaces. In Subsection \ref{global}  we construct the {\it global}  version of all these spaces, and discuss their properties and the properties of the relative exact sequence. In Section \ref{disc-pro} we first introduce the discretized problem, and in Subsection \ref{theo:est} we prove  the a priori  error bounds for it. In Section \ref{sec:NE} we present some numerical results that show that the quality of the approximation is very good, and also that the serendipity variant does not jeopardize the accuracy. 

\section{Notation and well known properties of polynomial spaces}\label{propol}

In two dimensions, we will denote by $\xx$  the indipendent variable, using $\xx=(x,y)$
or (more often) $\xx=(x_1,x_2)$ following the circumstances. We will also use
$\xx^{\perp}:=(-x_2,x_1)$,
and in general, for a vector $\vv\equiv(v_1,v_2)$,
\begin{equation}\label{vperp}
\vv^{\perp}:=(-v_2,v_1).
\end{equation}

\noindent
Moreover, for a vector $\vv$ and a scalar $q$ we will write
\begin{equation}\label{rotebrot}
\displaystyle{\rot \vv:=\frac{\partial v_2}{\partial x}-\frac{\partial v_1}{\partial y}},
\qquad \brot \,q:=(\frac{\partial q}{\partial y}, -\frac{\partial q}{\partial x})^T.
\end{equation}
We recall some commonly used functional spaces. On a  domain $\O$ we have
 \begin{align*}
&H(\div;\O)=\{\vv\in [L^2(\O)]^3 \mbox{ with } \div \vv \in L^2(\O)\},\\
&H_0(\div;\O)=\{\bphi\in H(\div;\O)\mbox { s.t. } \bphi\cdot\nn=0 \mbox{ on }\partial\O\},\\
&H(\curl;\O)=\{\vv\in [L^2(\O)]^3 \mbox{ with } \curl \vv \in [L^2(\O)]^3\},\\
&H_0(\curl;\O)=\{\vv \in H(\curl;\O) \mbox{ with } {\vv}\wedge\nn =0 \mbox{ on } \partial \O \},\\
&H^1(\O)=\{q\in L^2(\O) \mbox{ with } \bgrad\, q \in [L^2(\O)]^3\},\\
&H^1_0(\O)=\{q\in H^1(\O) \mbox{ with } q=0 \mbox{ on } \partial\O\}.
\end{align*}
\noindent For an integer $s\ge -1$ we will denote by $\P_s$ the space of polynomials of degree $\le s$.
Following a common convention, $\P_{-1}\equiv\{0\}$ and $\P_0\equiv\R$. Moreover, for
$s\ge 1$
 \begin{equation}\label{defPh}
\P^h_s:=\{ \mbox{homogeneous pol.s in $\P_s$}\},\quad\P^0_s(\O):=\{q\in\P_s \mbox{ s. t. }\!\int_{\O}q\,{\rm d}\O=0\}.
\end{equation}

The following decompositions of polynomial vector spaces are well known and will be useful in what follows. In two dimensions we have
\begin{equation}\label{decoPs}
(\P_s)^2=\brot (\P_{s+1})\oplus \xx \P_{s-1} \quad\mbox{ and }\quad(\P_s)^2=\bgrad (\P_{s+1})\oplus \xx^{\perp} \P_{s-1},
\end{equation}
and in three dimension
\begin{equation}\label{decompPs3D}
(\P_s)^3=\bcurl((\P_{s+1})^3) \oplus\xx\P_{s-1},\quad\mbox{ and }\quad
(\P_s)^3=\bgrad(\P_{s+1}) \oplus\xx\wedge(\P_{s-1})^3.
\end{equation}
Taking the  $\bcurl$ of the second of \eqref{decompPs3D} we also get :
\begin{equation}\label{decompPs3Dbis}
\bcurl(\P_s)^3= \bcurl (\xx\wedge(\P_{s-1})^3)
\end{equation}
which used in the first of \eqref{decompPs3D} gives:
\begin{equation}\label{decompPs3Dter}
(\P_s)^3= \bcurl (\xx\wedge(\P_{s})^3)\oplus\xx\P_{s-1}.
\end{equation}
We also recall the definition of the N\'ed\'elec {\it local} spaces of 1-st and 2-nd kind.
\begin{equation}\label{N1-N2-2D}
\begin{aligned}
&\mbox{In $2$d:}\quad N1_s=\bgrad\, \P_{s+1}\oplus \xx^{\perp}(\P_{s})^2, ~s\ge 0,\qquad
 N2_s:=(\P_s)^2, ~s\ge 1,\\
%
&\mbox{in $3$d:}\quad N1_s=\bgrad\, \P_{s+1}\oplus \xx\wedge(\P_{s})^3, ~s\ge 0,\qquad
 N2_s:=(\P_s)^3, ~s\ge 1. 
\end{aligned}
\end{equation}

In what follows, when dealing with the {\it faces} of a polyhedron (or of a polyhedral decomposition) we shall use two-dimensional differential operators that act on the restrictions to faces of scalar functions that are defined on a three-dimensional domain. Similarly, for vector valued functions we will use two-dimensional differential operators that act on the restrictions to faces of the tangential components.  In many cases, no confusion will be likely to occur; however, to stay on the safe side, we will often use a superscript $\tau$  to denote the tangential components of a three-dimensional vector, and a subscript $f$ to indicate the two-dimensional differential operator. Hence, to fix ideas, if a face has equation
$x_3=0$ then $\xx^{\tau}:=(x_1,x_2)$ and, say, $\div_f\vv^{\tau}:=\frac{\partial v_1}{\partial x_1}+
\frac{\partial v_2}{\partial x_2}$.

\section{The problem and the spaces}\label{discre1}

\subsection{The Kikuchi variational formulation}\label{kikvar}
Here we { shall} deal with the variational formulation introduced  in
\cite{KikuchiIEEE}, that reads
\begin{equation}\label{K1_3}
\left\{
\begin{aligned}
& \mbox {find  }\HH\in H_0(\bcurl;\Om) \mbox{ and }p\in H^1_0(\Om) \mbox{ such that: }\\
&\mbox{$ \int_{\Om}\bcurl\HH\!\cdot\!\bcurl\vv\dO+\!\int_{\Om}\nabla p\!\cdot\!\mu\vv\dO=\!\int_{\Om}\jj\!\cdot\!\bcurl\vv\dO
\quad\forall\vv\in H_0(\bcurl;\Om)$}\\
&\mbox{$ \int_{\Om}\nabla q\!\cdot\!\mu\HH\dO=0\quad\forall q\in H^1_0(\Om).$}\\
\end{aligned}
\right.
\end{equation}
It is easy to check
that \eqref{K1_3} has a unique solution $(\HH,p)$. Then we check that
$\HH$ and $\mu\HH$  give the solution of  \eqref{Max3}  and $p=0$.  Checking  that $p=0$ is immediate, just taking $\vv=\nabla p$ in the first equation. Once we know that $p=0$ the first equation gives $\bcurl\HH=\jj$, and then the second equation gives $\div\mu\HH=0$.

We will now design the Virtual Element approximation of \eqref{K1_3}  of order $k\ge 1$. We define first the local spaces. {Let  $\PP$ be a polyhedron,
simply connected, with all its faces also simply connected and convex. (For the treatment of non-convex faces we refer to \cite{SERE-mix}). More detailed assumptions will be given in Section \ref{theo:est}.}

\newcommand{\kd}{k}
\newcommand{\kr}{k-1}

\subsection{The  local spaces on faces}\label{local}

We first recall the local {\it nodal} and {\it edge} spaces on faces introduced in \cite{max2}. We shall deal with a sort of generalisation to polygons of {\it N\'ed\'elec elements of the second kind} $N2$ (see \eqref{N1-N2-2D}). For this, let $k \ge 1$. For each face $f$ of $\PP$, the {\it edge} space on $f$ is defined as
\begin{equation} \label{k-edgef}
V^{\edge}_{k}(f)\! :=\! \Big\{\! \vv\!\in\! [L^2(f)]^2\!:
\div\vv \!\in\! \P_{\kd}(f), \, \rot\vv\! \in\! \P_{\kr}(f),\,
\vv \cdot \tt_e\! \in\! \P_{k}(e) \, \forall e\! \in\! \partial f \Big\} ,
\end{equation}
with the degrees of freedom
\begin{align}
&\bullet
\mbox{ on each $e\in \partial f$, the moments $\int_e (\vv\cdot\tt_e) p_{k} \ds \quad \forall p_{k} \in \P_{k}(e) $,}
\label{k-2dofe1}\\
&\bullet
\mbox{  the moments $ \int_{f}\vv\cdot\xxf \: p_{\kd}\df \quad \forall p_{\kd} \in\P_{\kd}(f)$,} \label{k-2dofe2}\\
&\bullet\mbox{$\int_{f} \rot\vv ~ p_{\kr}^0 \df \quad \forall p_{\kr}^0 \in \P_{\kr}^0(f)$}\qquad\mbox{ (only for $k > 1$)}) , \label{k-2dofe3}
\end{align}
where $\xxf=\xx-{\bf b}_{f}$, with ${\bf b}_{f}=$  barycenter of $f$, and $\P_s^0$ was defined in \eqref{defPh}.

We recall that for $\vv\in V^{\edge}_{k}(f)$  the value of $\rot\vv$ is easily computable from the degrees of freedom {\eqref{k-2dofe1}} and
\eqref{k-2dofe3}. Indeed, the mean value of $\rot\vv$ on $f$ is computable from \eqref{k-2dofe1} and Stokes Theorem, and then
 (since $\rot\vv\in\P_{\kr}$) the use of \eqref{k-2dofe3} gives the full value of
 $\rot\vv$. Once we know $\rot\vv$,  following \cite{max2}, we can easily compute, always for each $\vv\in V^{\edge}_{k}(f)$, the $L^2$--projection $\Pi^0_{k+1}:V^{\edge}_{k}(f) \rightarrow [\P_{k+1}(f)]^2$. Indeed: by definition of projection, using \eqref{decoPs} and integrating by parts we obtain:
\begin{equation}\label{proj-edge-facce}
\begin{aligned}
\mbox{$\int_f \Pi^0_{k+1}\vv\cdot{\bf p}_{k+1}$}&\df:=\mbox{$\int_f \vv\cdot{\bf p}_{k+1} \df  = \int_f \vv\cdot (\text{\bf{rot}}\, q_{k+2} + \xxf q_{k} ) \df $} \\
& = \mbox{$  \int_f (\rot \vv) q_{k+2}  \df +  \sum_{e \in \partial f} \int_e (\vv\cdot\tt) q_{k+2} \ds
+ \int_f \vv\cdot\xxf q_{k}  \df$}
\end{aligned}
\end{equation}
and it is immediate to check that each of the last three terms is computable.

\begin{remark}\label{loc-scalp}
 Among other  things, projection operators can be used to define suitable scalar products in $V^{\edge}_k(f)$. As common in the virtual element literature, we could use the (Hilbert) norm
\begin{equation}\label{coidof-e}
\|\vv\|_{\epsf}^2:=\|\Pi^0_k\vv\|_{0,f}^2+\sum_i (dof_i\{(I-\Pi^0_k)\vv\})^2,
\end{equation}
where the $dof_i$ are the degrees of freedom in $V^{\edge}_k(f)$, properly scaled. In \eqref{coidof-e} we could also insert any symmetric and positive definite matrix $S$ and change the second term into ${\bf d}^T\,S\,{\bf d}$
(with ${\bf d}=$ the vector of the $dof_i\{(I-\Pi^0_k)\vv\}$). 
 Alternatively we could use
\begin{equation}\label{esc-bd}
\|\vv\|_{\epsf}^2:=\|\Pi^0_{k+1}\vv\|_{0,f}^2+h_f\|(I-\Pi^0_{k+1})\vv\cdot \tt\|_{0,\partial f}^2
\end{equation}
(that is clearly a Hilbert norm) where $h_f$ is the diameter of the face $f$. It is easy to check that the associated inner product {\it scales} like the natural $[L^2(f)]^2$ inner product (meaning that  $\|\vv\|_{V^{\edge}_k(f)}$ is bounded above and below by $\|\vv\|_{0,f}$ times suitable constants independent of $h_f$), and moreover coincides with the $[L^2(f)]^2$ inner product whenever one of the two entries is in $(\P_{k+1})^2$.    \qed
\end{remark}
For each face $f$ of $\PP$,  the {\it nodal} space of order $k+1$ is defined as
\begin{equation}\label{k-nodf}
 V_{k+1}^{\node}(f)    := \Big\{ q \in H^1(f) : \: q_{|e} \in \P_{k+1}(e) \: \forall e \in \partial f, \: \Delta q \in \P_{\kd}(f) \Big\},
 \end{equation}
with the degrees of freedom
\begin{align}
& \bullet \mbox{ for each vertex $\nu$ the value $q(\nu)$, } \label{k-2dofn0}\\
& \bullet\mbox{ for each edge $e$ the moments $\int_e q\, p_{k-1}\ds\quad\forall p_{k-1}\in\P_{k-1}(e)$},\label{k-2dofn1}\\
& \bullet\,
\mbox{$\int_{f}(\nabla q \cdot \xxf)\: p_{\kd} \df \quad \forall p_{\kd} \in\P_{\kd}(f).$}\label{k-2dofn2}
\end{align}

\subsection{The local serendipity spaces on faces}\label{sere}

We recall the serendipity spaces introduced in \cite{max2}, which will be used to construct the serendipity spaces on polyhedra. Let $f$ be a face of $\PP$, assumed to be a convex polygon. Following \cite{SERE-nod} we introduce 
\begin{equation}\label{defbeta}
\beta :=k+1-\eta .
\end{equation}
where $\eta$ is {\it the number of straight lines necessary to cover the boundary of $f$}.
We note that the convexity of $f$  does not imply that $\eta$ is equal to the number of edges of $f$, since we might have different consecutive edges that belong to the same straight line.
Next, we
define a projection $\Pi_{S}^{\edge}: V_{k}^{\edge} (f) \rightarrow \P_{k}(f)$ as follows:
\begin{align}
&\mbox{$\int_{\partial f}[(\vv - \Pi_{S}^{\edge} \vv)\cdot\tt] [\nabla p\cdot\tt]  \ds= 0 \quad \forall { p} \in \P_{k+1}(f)$},\label{defPieS1}\\
&\mbox{$\int_{\partial f}(\vv - \Pi_{S}^{\edge} \vv)\cdot \tt \ds=0$},\label{defPieS2}\\
&\mbox{$\int_f \rot(\vv-\Pi^{\edge}_S\vv)p^0_{\kr} \df=0\quad \forall p_{\kr}^0 \in \P_{\kr}(f)\quad \mbox{for } k> 1$} ,\label{defPieS3}\\
&\mbox{$\int_{f}(\vv-\Pi^{\edge}_S\vv)\cdot\xxf \: p_{\beta}\df \quad \forall p_{\beta} \in\P_{\beta}(f)\quad \mbox{{\bf only} for } \beta \ge 0.$}\label{defPieS4}
\end{align}
The {\bf serendipity edge space} is then defined as:
\begin{equation}\label{defVeS}
SV_{k}^{\edge}(f) : = \Big\{ \vv \in V_{k}^{\edge}(f)  \: : \: \int_{f} (\vv-\Pi_{S}^{\edge} \vv )  \cdot \xxf\,p \df=0 \quad\forall p\in \P_{\beta|\kd}(f)\Big\} ,
\end{equation}
where $\P_{\beta|\kd}$ is the space spanned by { all} the homogeneous polynomials of degree $s$ with $\beta< s\le \kd$.
The degrees of freedom in $SV_{k}^{\edge}(f)$ will be \eqref{k-2dofe1} and \eqref{k-2dofe3}, plus
\begin{equation}\label{k-2dofe3-sere}
\int_{f}\vv\cdot\xxf \: p_{\beta}\df \quad \forall p_{\beta} \in\P_{\beta}(f) \qquad \mbox{ \bf only if } \beta\ge 0.
\end{equation}
 To summarize: {if} $\beta<0$, i.e., if $k+1<\eta$, the only internal degrees of freedom are \eqref{k-2dofe3}, and the moments \eqref{k-2dofe2} are given by those of $\Pi_{S}^{\edge}$. Instead, for $\beta \ge 0$ we have to include among the d.o.f. the moments of order up to $\beta$ given in \eqref{k-2dofe3-sere}. The remaining moments, of order up to $k$, are again  given by those of $\Pi_{S}^{\edge}$. We point out that, on triangles, these are now exactly the N\'ed\'elec elements of second kind.

{Clearly in $SV^{\edge}_k(f)$ (that is included in $V^{\edge}_k(f)$) we can still use the scalar product defined in \eqref{esc-bd} or \eqref{coidof-e}}.

For the construction of the {\it nodal} serendipity space we proceed as before. Let $\Pi^{\node}_{S}: V_{k+1}^{\node}(f) \rightarrow \P_{k+1}(f)$ be a projection  defined by
\begin{equation}\label{PiS-nod2}
\left\{
\begin{aligned}
&\mbox{$\int_{\partial f} \partial_t (q-\Pi^{\node}_{S}q)\partial_t p \ds=0\quad \forall p\in \P_{k+1}(f)$},\\
&\mbox{$ \int_{\partial f} (\xxf\cdot\nn) (q - \Pi_S^{\node} q) \ds = 0$} ,\\
&
\mbox{$\int_{f}(\nabla(q-\Pi^{\node}_S q)\cdot\xxf\, p_{\beta}\df=0\quad\forall p_{\beta}\in\P_{\beta}$}\quad \mbox{{\bf only} for } \beta \ge 0.
\end{aligned}
\right.
\end{equation}
The {\bf serendipity nodal space} is then defined as:
\begin{equation}\label{defVnS}
SV_{k+1}^{\node}(f) := \Big\{ q \in V_{k+1}^{\node}(f)  \, : \: \int_{f} (\nabla q - \nabla\Pi_{S}^{\node} q )\cdot \xxf\,p\df=0 \,\forall p\in \P_{\beta| \kd}(f) \Big\} .
\end{equation}
The degrees of freedom in $SV_{k+1}^{\node}(f)$ will be \eqref{k-2dofn0} and \eqref{k-2dofn1}, plus
\begin{equation}\label{k-2dofn2-sere}
\int_{f}(\nabla q \cdot \xxf)\: p_{\beta} \df \quad \forall p_{\beta} \in\P_{\beta}(f) \qquad \mbox{ \bf only if } \beta\ge 0.
\end{equation}
From this construction it follows that the nodal serendipity space contains internal d.o.f. only if $k+1\ge \eta$, and the number of these d.o.f. is equal to the dimension of $\P_{\beta}$ only. The remaining d.o.f. are copied from those of $\Pi_{S}^{\node}$. Note also that on triangles we have back the old polynomial Finite Elements of degree $k+1$. Before dealing with the three dimensional spaces, we recall a useful result proven in \cite{max2}, Proposition 5.4.
\begin{prop}\label{2d-diagram-Vn-Ve}
It holds
\begin{equation}\label{2d-Vn-Ve}
\nabla SV^{\node}_{k+1}(f)=\{ \vv \in SV^{\edge}_{k}(f):~ \rot \vv ={0}\}.
\end{equation}
\end{prop}
\newcommand{\kdP}{k-1}
The following result is immediate, but we point it out for future use.

\begin{prop}\label{bubbleface}
For every $q\in V_{k+1}^{\node}(f)$ there exists a (unique) $q^*$  such that
\begin{equation}\label{sigmaq}
q^*\in SV_{k+1}^{\node}(f)\quad\mbox{(and we denote it as $q^*=\sigma^{{\node},f}(q)$)},
\end{equation}
that has the same degrees of freedom \eqref{k-2dofn0},\eqref{k-2dofn1}, and \eqref{k-2dofn2-sere} of $q$. The difference $q-q^*$ is obviously a bubble in $V_{k+1}^{\node}(f)$.
Similarly,
for a $\vv$ in $V_{k}^{\edge}(f)$ there exists a unique $\vv^*$ with 
\begin{equation}\label{sigmav}
\vv^*\in SV_{k}^{\edge}(f)\quad\mbox{(and we denote it as $\vv^*=\sigma^{{\edge},f}(\vv)$)},
\end{equation}
with the same degrees of freedom \eqref{k-2dofe1}-\eqref{k-2dofe3}, and \eqref{k-2dofe3-sere} of $\vv$.  The difference $\vv-\vv^*$ is an $H(rot)$-bubble and, in particular, is the gradient of a scalar  bubble
$\xi(\vv)$:
\begin{equation}\label{xiv}
\nabla \xi\equiv \vv-\vv^*
\end{equation}
\end{prop}
\begin{proof} It is clear from the previous discussion that the degrees of freedom \eqref{k-2dofn0}, \eqref{k-2dofn1}, and \eqref{k-2dofn2-sere} determine $q^*$ in a unique way.
As $q$ and $q^*$ share the same boundary degrees of freedom \eqref{k-2dofn0}
 and   \eqref{k-2dofn1}, they will coincide on the whole boundary $\partial f$, so that $q-q^*$ is a bubble. Similarly,
given $\vv$ in $V_{k}^{\edge}(f)$ the degrees of freedom \eqref{k-2dofe1}-\eqref{k-2dofe3}, and \eqref{k-2dofe3-sere} determine uniquely a $\vv^*$ in $SV_{k}^{\edge}(f)$. The two vector valued functions $\vv$ and $\vv^*$, sharing the degrees of freedom \eqref{k-2dofe1}-\eqref{k-2dofe3} must  have the same tangential components on $\partial f$ and
{\it the same $\rot$}. In particular, $\rot(\vv-\vv^*)=0$ and (as $f$ is simply connected)
$\vv-\vv^*$ must be a gradient of some scalar function $\xi$ (that  we can take as a bubble, since its tangential derivative
on $\partial f$ is zero).
\end{proof}

\subsection{ The local spaces on polyhedra}\label{local3}

Let  $\PP$ be a polyhedron, simply connected with all its faces simply connected and convex.
 For each face $f$ we will use the  serendipity spaces $SV^{\node}_{k+1}(f)$ and $SV^{\edge}_{k}(f)$ as defined in \eqref{defVnS} and \eqref{defVeS}, respectively.
We then introduce the three-dimensional  analogues of \eqref{defVnS} and \eqref{defVeS}, that are
 \begin{multline} \label{edge3}
V^{\edge}_{k}(\PP) := \Big\{ \vv\in [L^2(\PP)]^3\,:\div\vv\in \P_{\kdP}(\PP), \: \curl(\curl\vv) \in [\P_{k}(\PP)]^3,\\
\vv^{\tau}_{|f}\in SV^{\edge}_{k}(f)~\forall \mbox{ face }f\in\partial\PP,\; \vv\cdot\tt_e \mbox{ continuous on each edge } e\in\partial\PP
 \Big\} ,
\end{multline}
\begin{equation}\label{nod3}
 V^{\node}_{k+1}(\PP):= \Big\{ q \in C^0(\PP) \: : q_{|f}\in SV^{\node}_{k+1}(f)\quad\forall \mbox{ face }f\in\partial \PP, \:
  \,\Delta\,q\in\P_{\kdP}(\PP)\Big\}.
 \end{equation}
This time however we will also need a Virtual Element {\it face} space (for the discretization of two-forms), that we define as
\begin{equation} \label{face3}
\! {V}^{\face}_{k-1}(\PP) \!:=
\! \Big\{ \ww\!\in\![L^2(\PP)]^3: {\div\ww\!\in\! \P_{k-1}}, \, \curl\ww\!\in\!\! [\P_{k}]^3,\:   \ww\cdot\nn_f\!\in\!\P_{k-1}(f)~\forall f \Big\}.
\end{equation}
We note that in several { cases, in particular for polyhedra with many faces}, the number of {\it internal} degrees of freedom
for the spaces  \eqref{edge3}, \eqref{nod3}, and \eqref{face3} will be {\it more than necessary}. However, at this point, we will not make efforts to diminish them, as we
assume that in practice we could eliminate them by static condensation (or even construct suitable serendipity variants).

Among the same lines of Proposition \ref{bubbleface}, we have now:

\begin{prop}\label{bubblepol}
For every
function $q$ in the (non serendipity!) space
 \begin{equation} \label{node3tilde}
\widetilde{V}^{\node}_{k+1}(\PP) := \Big\{q\in C^0(\PP)\,:\, q_{|f}\in V^{\node}_{k+1}(f)~\forall \mbox{ face }f\in\partial\PP,\mbox{ and } \Delta q\in\P_{k-1}\Big\}
\end{equation}
there exists exactly one element $q^*=\sigma^{\node,\PP}(q)$ in ${V}^{\node}_{k+1}(\PP)$ such that 
\begin{equation}
q^*_{|f}=\sigma^{{\node},f}(q_{|f}) \qquad \forall \mbox{ face } f,\quad
\mbox{and }\qquad\Delta (q-q^*)=0 \mbox{ in }\PP.
\end{equation}
Similarly, for every vector-valued
function $\vv$ in the (non serendipity!) space
 \begin{multline} \label{edge3tilde}
\widetilde{V}^{\edge}_{k}(\PP) := \Big\{ \vv\in [L^2(\PP)]^3\,:\div\vv\in \P_{\kdP}(\PP), \: \curl(\curl\vv) \in [\P_{k}(\PP)]^3\\
\vv^{\tau}_{|f}\in V^{\edge}_{k}(f)~\forall \mbox{ face }f\in\partial\PP,\; \vv\cdot\tt_e \mbox{ continuous on each edge } e\in\partial\PP
 \Big\} ,
\end{multline}
there exists exactly one element $\vv^*=\sigma^{\edge,\PP}(\vv)$ of ${V}^{\edge}_{k}(\PP)$ such that:
\begin{align}
&\bullet \quad\mbox{ on each face $f$ of }\partial\PP: \quad(\vv^*)^{\tau}=\sigma^{{\edge},f}(\vv^{\tau})\quad\mbox{(as defined in \eqref{sigmav})}, \label{vonB}\\[2mm]
%
&\bullet \quad\mbox{ and in }\PP:\qquad\div(\vv-\vv^*)=0\quad\mbox{ and }\quad \bcurl(\vv-\vv^*)={\bf 0} .\label{vinside}
\end{align}
\end{prop}
\begin{proof} The first part, relative to nodal elements, is obvious: on each face we take as $q^*$ the one given by \eqref{sigmaq} in Proposition \ref{bubbleface}, and then we take $\Delta q^*=\Delta q$ inside. For constructing $\vv^*$ we also start by defining its tangential components on each
face using Proposition \ref{bubbleface}. Now, on each face $f$ we have a (scalar) bubble
$\xi_{f}$ (whose tangential gradient equals the tangential components of $\vv-\vv^*)$, and we construct in $\PP$ the scalar function $\xi$ which is: equal to $\xi_{f}$ on each face $f$, and  harmonic inside $\PP$. Then we set $\vv^*:=\vv+\nabla\xi$, and we check immediately  that $\vv^*$ verifies property \eqref{vonB},  and also properties  \eqref{vinside}, since $\xi$ vanishes on all edges and is harmonic inside.
\end{proof}

\begin{prop}\label{exactN-E}
It holds
\begin{equation}\label{diagramVn-Ve}
\nabla V^{\node}_{k+1}(\PP)=\{ \vv \in V^{\edge}_{k}(\PP):~ \bcurl \vv ={\bf 0}\}.
\end{equation}
\end{prop}
\begin{proof}
From the above definitions we easily see that the  {\it tangential gradient} of any $q\in V^{\node}_{k+1}(\PP)$, applied {\it face by face}, belongs to $SV^{\edge}_{k}(f)$. Consequently, we also have that $\vv:=\bgrad q$ belongs to $ V^{\edge}_{k}(\PP)$, as $\div\vv\in\P_{\kdP}(\PP)$ and $\bcurl\vv=0$. Hence,
\begin{equation}
\nabla V^{\node}_{k+1}(\PP) \subseteq \{ \vv \in V^{\edge}_{k}(\PP):~ \bcurl \vv ={\bf 0}\}.
\end{equation}
Conversely, assume that a $\vv \in V^{\edge}_{k}(\PP)$ has $ \bcurl \vv ={\bf 0}$. As $\PP$ is simply connected we  have that $\vv=\nabla q$  for some $q\in H^1(\PP)$.  On each face
$f$, the tangential gradient of $q$ (equal to $\vv^{\tau}$) is in $SV^{\edge}_k(f)$ (see \eqref{edge3}), and since  $\rot_f\vv^{\tau}=\bcurl \vv\cdot \nn_f\equiv 0$, from \eqref{2d-Vn-Ve}  we deduce that $q_{|f} \in SV^{\node}_{k+1}(f)$. Hence, the restriction of $q$ to the boundary of $\PP$ belongs to $V^{\node}_{k+1}(\PP)_{|\partial \PP}$. Moreover, $\Delta q=\div \vv$ is in $ \P_{k-1}(\PP)$. Hence, $q\in V^{\node}_{k+1}(\PP)$ and  the proof is concluded.
\end{proof}

In
$V^{\edge}_{k}(\PP) $ we have (see \cite{max2} and \cite{SERE-mix}) the  degrees of freedom
\begin{align}
& \bullet\mbox{ $\forall$  edge $e$: $\int_e (\vv\cdot\tt_e) p_{k}\ds \quad \forall p_{k} \in \P_{k}(e) ,$ } \label{dof-3dek-1}\\
& \bullet \mbox{ $\forall$  face $f$ with $\beta_f\ge 0: \int_{f}\vv^{\tau}\cdot\xxf \: p_{\beta_f}\df \quad \forall p_{\beta_f} \in\P_{\beta_f}(f),$} \label{dof-3dek-2}\\
& \bullet \mbox{ $\forall$ face $f$:  $\int_{f} \rot_f\vv^{\tau} \, p_{k-1}^0 \df \quad \forall p_{k-1}^0 \in \P_{k-1}^0(f)$} \quad\mbox{(for $k > 1$)}, \label{dof-3dek-3}\\
& \bullet \mbox{ $ \int_{\PP}
(\vv\cdot\xx_{\PP}) p_{\kdP}\, \dPP \quad \forall p_{\kdP} \in \P_{\kdP}(\PP)$} ,\label{dof-3dek-4} \\
& \bullet\mbox{ $\int_{\PP} 
(\bcurl\vv)\cdot(\xx_{\PP}\wedge {\bf p}_{k})\,\dPP \quad \forall {\bf p}_{k} \in [\P_{k}(\PP)]^3$},  \label{dof-3dek-5}
\end{align}
where  $\beta_f$ = value of $\beta$ (see  \eqref{defbeta}) on $f$, and
$\xx_{\PP\!}:=\!\xx\! -\!{\bf b}_{\PP}$, with ${\bf b}_{\PP}\!= $barycenter of $\PP$.
\begin{prop}
Out of the above degrees of freedom we can compute the $[L^2(\PP)]^3$  orthogonal projection $\Pi^{0}_{k}$ from $V^{\edge}_{k}(\PP)$ to $[\P_k(\PP)]^3$.
\end{prop}
\begin{proof}
Extending the arguments used in \cite{lowest-max3},  and  using \eqref{decompPs3Dter} we have that for any $\pp_k\in (\P_k)^3$ there exist two polynomials, ${\qq}_k\in (\P_k)^3$ and $z_{k-1}\in \P_{k-1}$, such that $\pp_k=\curl(\xxP \wedge {\bf q}_k)+\xxP z_{k-1}$. Hence,  from the definition of projection we have:
\begin{equation}\label{proj-edge-3D}
\int_{\PP} \Pi^{0}_k\vv \cdot \pp_k \dPP:=\int_{\PP} \vv \cdot \pp_k \dPP
=\int_{\PP} \vv \cdot \curl(\xxP \wedge {\qq}_k) \dPP+\int_{\PP}(\vv\cdot\xxP) z_{k-1} \dPP.
\end{equation}
The second integral is given by the d.o.f. \eqref{dof-3dek-4}, while for the first one we have, upon integration by parts:
\begin{equation}\label{proj-edge-3D-a}
\begin{aligned}
\mbox{$\int_{\PP} \vv \cdot$}&\mbox{$\curl(\xxP \wedge \qq_k) \dPP=
{\int_{\PP} \bcurl\vv \!\cdot\! (\xxP \wedge \qq_k) \dPP} \!+\! \int_{\partial \PP} (\vv\wedge \nn) \!\cdot \!(\xxP \wedge{\qq}_k) \dS$}\\
&=\mbox{$\int_{\PP} \bcurl\vv \cdot (\xxP \wedge \qq_k) \dPP+{\int_{\partial \PP} \Big(\nn \wedge(\xxP \wedge\qq_k)\Big) \cdot {\vv}  \dS}$}\\
&=\mbox{$\int_{\PP} \bcurl\vv \cdot (\xxP \wedge \qq_k) \dPP+{\sum_f \int_f \Big(\nn_f \wedge(\xxP \wedge\qq_k)\Big)^{\tau} \cdot {\vv^{\tau}}  \df.}$}
\end{aligned}
\end{equation}
The first term is given by the d.o.f. \eqref{dof-3dek-5}, and the second is computable as in \eqref{proj-edge-facce}.
\end{proof}
 Hence, following the path of Remark \ref{loc-scalp} we can define
a $\mu$-dependent scalar product through the (Hilbert) norm
\begin{equation}\label{norma-nuova}
\|\vv\|^2_{\epsmP}:=\|\mu^{1/2}\Pi^{0}_{k}\vv\|^2_{0,\PP}+h_{\PP}\mu_0\sum_i (dof_i\{(I-\Pi^0_k)\vv\})^2,
\end{equation}
or, for instance,
\begin{equation}\label{norma1}
\|\vv\|^2_{\epsmP}:=\|\mu^{1/2}\Pi^{0}_{k}\vv\|^2_{0,\PP}+h_{\PP}\mu_0\sum_{f\in\partial\PP}\|(I-\Pi^0_k)\vv^{\tau}\|_{V^{\edge}_k(f)}^2
\end{equation}
getting, for positive constants $\alpha_*, \alpha^*$ independent of $h_{\PP}$,
\begin{equation}\label{PSe3k-1}
\alpha_* \mu_0 \|\vv\|_{0,\PP}^2\le \|\vv\|_{e,\mu,\PP}^2\le
\alpha^* \mu_1\|\vv\|_{0,\PP}\qquad\forall \vv\in V^{\rm{e}}_{k}(\PP) .
\end{equation}
We observe that the associated scalar product will satisfy
\begin{equation}\label{stima-edge}
[\vv,\ww]_{0,\PP}\le\Big( [\vv,\vv]_{\epsmP}\Big)^{1/2}\Big( [\ww,\ww]_{\epsmP}\Big)^{1/2}\le\mu_1 \alpha^* \|\vv\|_{0,\PP} \|\ww\|_{0,\PP} ,
\end{equation}
\begin{equation}\label{consiE3k}
\mbox{$[\vv,{\bf p}_{k}]_{\epsmP}=\int_{\PP}\mu\Pi^{0}_{k}\vv\cdot {\bf p}_{k}\dE \qquad\forall \vv\in V^{\edge}_{k}(\PP),\;\forall {\bf p}_{k}\in [\P_{k}(\PP)]^3$}.
\end{equation}

In $V^{\node}_{k+1}(\PP)$ we have the degrees of freedom
\begin{align}
& \bullet \mbox{ $\forall$ vertex $\nu$ the nodal value $q(\nu),$ } \label{dof-3dnk-1}\\
& \bullet \mbox{ $\forall$ edge $e$ and $k\ge 1$ the moments $ \int_e q\, p_{k-1}\ds\quad\forall p_{k-1}\in\P_{k-1}(e),$}\label{dof-3dnk-2}\\
&\bullet
\mbox{ $\forall$ face $f$  with $\beta_f\ge 0$ the moments $\int_{f}(\nabla_f q\cdot \xx_f)\: p_{\beta_f} \df \quad \forall p_{\beta_f} \in\P_{\beta_f}(f)$,}\label{dof-3dnk-3}\\
&\bullet\mbox{ the moments$\int_{\PP} q \: p_{\kdP} \dPP \quad \forall p_{\kdP} \in\P_{\kdP}(\PP).$}\label{dof-3dnk-4}
\end{align}

%
We point out (see \cite{max2}) that the  degrees of freedom \eqref{dof-3dnk-1}-\eqref{dof-3dnk-3} on each face $f$ allow to compute the $L^2(f)$-orthogonal projection operator from $SV^{\node}_{k+1}(f)$ to $\P_{k}(f)$, while the degrees of freedom \eqref{dof-3dnk-4} give us  the $L^2(\PP)$-orthogonal projection operator from $V^{\node}_{k+1}(\PP)$ to $\P_{\kdP}(\PP)$.
Finally, for $V^{\face}_{k-1}(\PP)$ we have the degrees of freedom
\begin{align}
& \bullet\mbox{ $\forall$  face $f$: $\int_f
(\ww\cdot\nn_f) p_{k-1} \df \quad \forall p_{k-1} \in \P_{k-1}(f), $} \label{dof-3dfk-1}\\
& \bullet \mbox{$\int_{\PP}
{\ww\cdot(\bgrad\, p_{k-1}) \dPP \quad \forall p_{k-1} \in \P_{k-1}(\PP), \mbox{ for } k >1 }$}\label{dof-3dfk-2} \\
& \bullet \mbox{$
\int_{\PP}
\ww\cdot (\xx_{\PP}\wedge {\bf p}_{k}) \dPP \quad \forall {\bf p}_{k} \in [\P_{k}(\PP)]^3$}.
 \label{dof-3dfk-3}
\end{align}
{
According to \cite{SERE-mix} we have now that from the above degrees of freedom we can compute the $[L^2(\PP)]^3$-orthogonal projection $\Pi^{0}_{s}$ from $V^{\face}_{k-1}(\PP)$ to $[\P_s(\PP)]^3$ with $s\le k+1$.

 In particular, along the same lines of Remark \ref{loc-scalp} we can define a scalar product $[\ww,\vv]_{\fpsP}$ through the Hilbert norm
\begin{equation}\label{PSf3k}
\|\vv\|_{\fpsP}^2:=\|\Pi^{0}_{k-1}\vv\|_{0,\PP}^2+h_{\PP}\sum_{f}\|(I-\Pi^{0}_{k-1})\vv\cdot \nn_f\|_{0,f}^2,
\end{equation}
and then there exist two positive constants
$\alpha_1,\,\alpha_2$ independent of $h_{\PP}$ such that
\begin{equation}\label{PSf3k-1}
\alpha_1 \|\ww\|_{0,\PP}^2\le \|\ww\|_{\fpsP}^2\le
\alpha_2 \|\ww\|_{0,\PP}^2\qquad\forall \ww\in V^{\face}_{k-1}(\PP),
\end{equation}
and also
\begin{equation}\label{consiF3k}
[\ww,{\bf p}_{k-1}]_{\fpsP}=(\ww, {\bf p}_{k-1})_{0,\PP}\qquad\forall \ww\in V^{\face}_{k-1}(\PP),\;\forall{\bf p}_{k-1}\in [\P_{k-1}(\PP)]^3.
\end{equation}}
Needless to say, instead of  \eqref{PSf3k} we could also consider variants of the type of \eqref{coidof-e} and \eqref{norma-nuova}, using only the values $dof_i$ of the degrees of freedom.

 Note that $\P_{k+1}(\PP)\subseteq V_{k+1}^{\node}(\PP)$, $[\P_{k}(\PP)]^3\subseteq V^{\edge}_{k}(\PP)$, and  $[\P_{k-1}(\PP)]^3\subseteq V^{\face}_{k-1}(\PP)$.
\begin{prop}\label{exact-Ve-Vf}
It holds:
\begin{equation}\label{diagramVe-Vf}
\bcurl V^{\edge}_{k}(\PP) = \{ \ww \in V^{\face}_{k-1}(\PP):~ \div \ww=0\}.
\end{equation}
\end{prop}
\begin{proof}
For every $\vv\in V^{\edge}_{k}(\PP)$ we have that $\ww:=\bcurl\, \vv$ belongs to $V^{\face}_{k-1}(\PP)$. Indeed, on each face $f$ we have that $\ww\cdot\nn_f\equiv\rot_f\vv^{\tau}$
belongs to $\P_{\kr}(f)$ (see \eqref{k-edgef} and \eqref{face3}), and moreover $\div\ww=0$ (obviously)
and $\bcurl\ww\in[\P_{k}(\PP)]^3$ from \eqref{edge3}. Hence,
\begin{equation}\label{added:sub}
\bcurl V^{\edge}_{k}(\PP)\subseteq \{ \ww \in V^{\face}_{k-1}(\PP):~ \div \ww=0\}.
\end{equation}

In order to prove the converse, we first note that from \cite{super-misti} we have that:  if
$\ww$ is in $V^{\face}_{k-1}(\PP)$ with $\div \ww=0$, then $ \ww=\bcurl\vv$ for some $\vv\in\widetilde{V}^{\edge}_{k}(\PP)$ (as defined in \eqref{edge3tilde}). Then we use Proposition \ref{bubblepol} and obtain a $\vv^*\in {V}^{\edge}_{k}(\PP)$ that, according to \eqref{vinside}, has the same $\bcurl$. An alternative proof could be derived by a simple dimensional count, following the same guidelines as in \cite{lowest-max3}.
\end{proof}

\subsection{The global spaces}\label{global}

Let $\Th$ be a decomposition of the computational domain $\Om$
into polyhedra $\PP$. On $\Th$ we make the following assumptions, quite standard in the VEM literature. We assume the existence of a positive constant $\gamma$ such that any  polyhedron $\PP$ of the mesh (of diameter $h_\PP$) satisfies the following conditions:
\begin{equation}\label{mesh-ass}
\begin{array}{ll}
&- \PP \mbox{ is star-shaped with respect to a ball of radius bigger than }\gamma h_\PP;\\
&-  \mbox{each face } f  \mbox{ is star-shaped with respect to a ball of radius  }
\ge \gamma h_P, \\
&- \mbox{each edge  has length bigger than $\gamma h_\PP$}.
\end{array}
\end{equation}
We note that the first two conditions imply that $\PP$ (and, respectively, every face of $\PP$) is simply connected.  At the theoretical level, some of the above conditions could be avoided by using more technical arguments.
We also point out that, at the practical level, as shown by the numerical tests of the Section \ref{sec:NE}, the third condition is negligible since the method seems very robust to degeneration of faces and edges.
On the contrary, although the scheme is quite robust to distortion of the elements, the first condition is more relevant since extremely anisotropic element shapes can lead to poor results. Finally, as already mentioned, for simplicity we also assume that all the faces are convex.

We can now define the global nodal space:
\begin{equation}\label{glo-n}
V^{\node}_{k+1}\equiv V^{\node}_{k+1}(\Om):=\Big\{ q \in H^1_0(\Om) \mbox{ such that } q_{|\PP} \in V^{\node}_{k+1}(\PP)\: \forall \PP \in \Th \Big\},
\end{equation}
 with the obvious degrees of freedom
\begin{align}
& \bullet \mbox{ $\forall$ vertex $\nu$ the nodal value $q(\nu),$ } \label{dof-3dnk-1G}\\
& \bullet \mbox{ $\forall$ edge $e$ and $k\ge 1$ the moments$\; \int_e q\, p_{k-1}\ds\quad\forall p_{k-1}\in\P_{k-1}(e),$}\label{dof-3dnk-2G}\\
&\bullet
\mbox{ $\forall$ face $f$  with $\beta_f\ge 0$ the moments $\int_{f}(\nabla_f q\cdot \xx_f)\: p_{\beta_f} \df \quad \forall p_{\beta_f} \in\P_{\beta_f}(f)$,}\label{dof-3dnk-3G}\\
&\bullet\mbox{ $\forall$ element $\PP$, $k\ge 1$, the moments$\int_{\PP} q \: p_{\kdP} \dPP \quad \forall p_{\kdP} \in\P_{\kdP}(\PP).$}\label{dof-3dnk-4G}
\end{align}
For the global edge space we have
\begin{equation}\label{glo-e}
V^{\edge}_{k}\equiv V^{\edge}_{k}(\Om):=\Big\{ \vv \in H_0(\bcurl;\Om) \mbox{ such that } \vv_{|\PP} \in V^{\edge}_{k}(\PP)\: \forall \PP \in \Th \Big\},
\end{equation}
with the obvious degrees of freedom
\begin{align}
& \bullet \mbox{ $\forall$ edge $e: \int_e (\vv\cdot\tt_e) p_{k}\ds \quad \forall p_{k} \in \P_{k}(e) ,$} \label{dof-3dek-1G}\\
& \bullet \mbox{ $\forall$  face $f$ with $\beta_f\ge 0: \int_{f}\vv^{\tau}\cdot\xxf \: p_{\beta_f}\df \quad \forall p_{\beta_f} \in\P_{\beta_f}(f),$} \label{dof-3dek-3G}\\
& \bullet \mbox{ $\forall$ face $f$:  $\int_{f} \rot_f\vv^{\tau} \, p_{k-1}^0 \df \quad \forall p_{k-1}^0 \in \P_{k-1}^0(f)$} \qquad
\mbox{(for  $k >1$)}, \label{dof-3dek-2G}\\
& \bullet\mbox{ $\forall$ element $\PP: \int_{\PP}
(\vv\cdot\xx_{\PP}) p_{\kdP}\, \dPP \quad \forall p_{\kdP} \in \P_{\kdP}(\PP) ,$}\label{dof-3dek-4G} \\
& \bullet\mbox{ $\forall$ element $\PP:\int_{\PP}
(\bcurl\vv)\cdot(\xx_{\PP}\wedge {\bf p}_{k})\,\dPP \quad \forall {\bf p}_{k} \in [\P_{k}(\PP)]^3$}.  \label{dof-3dek-5G}
\end{align}
 Finally, for the global face space we have:
\begin{equation}\label{globf}
V^{\face}_{k-1}\equiv V^{\rm{f}}_{k-1}(\Om):=\Big\{\ww\in H_0(\div;\Om)\mbox{ such that }\ww_{|\PP}\in V^{\face}_{k-1}(\PP)\:\forall \PP \in \Th\Big\},
\end{equation}
with the degrees of freedom
\begin{align}
& \bullet \mbox{ $\forall$ face $f: \int_f
(\ww\cdot\nn) p_{k-1} \df \quad \forall p_{k-1} \in \P_{k-1}(f), $} \label{dof-3dfk-1G}\\
& \bullet\mbox{ $\forall$ element $\PP:\,
\int_{\PP}
\ww\cdot (\xx_{\PP}\wedge {\bf p}_{k}) \dPP \quad \forall {\bf p}_{k} \in [\P_{k}(\PP)]^3$} , \label{dof-3dfk-2G }
\\
& \bullet\mbox{ $\forall$ element $\PP:~\int_{\PP}
\ww\cdot(\bgrad p_{k-1}) \dPP \quad \forall p_{k-1} \in \P_{k-1}(\PP)\quad k>1.$} \label{dof-3dfk-3G}
\end{align}
It is important to point out that
\begin{equation}\label{inclu3nek}
\nabla V^{\node}_{k+1}\subseteq V^{\edge}_{k}.
\end{equation}
In particular, it is easy to check that from Propositiom \ref{exactN-E} it holds
\begin{equation}\label{rot03k}
\nabla V^{\node}_{k+1}\equiv\{ \vv\in V^{\edge}_{k}\mbox { such that }\bcurl\vv=0\} .
\end{equation}
Similarly, also recalling Proposition \ref{exact-Ve-Vf}, we easily have
\begin{equation}\label{inclu3efk}
\bcurl V^{\edge}_{k}\subseteq V^{\face}_{k-1}.
\end{equation}
For the converse we  follow the same arguments of the proof of Proposition \ref{exact-Ve-Vf}: first using \cite{super-misti}, this time for the global spaces, and then correcting $\vv$ with a $\nabla \xi$ which is single-valued on the faces. Hence
\begin{equation}\label{rot03fk}
\bcurl V^{\edge}_{k}\equiv\{ \ww\in V^{\face}_{k-1}\mbox { such that }\div\ww=0\} .
\end{equation}
Introducing the additional space (for {\it volume} 3-forms)
\begin{equation}\label{glo-v}
V^{\vol}_{k-1}:=\{\gamma\in L^2(\Om)\;\mbox{ such that }\gamma_{|\PP}\in \P_{k-1}(\PP)\;\forall\PP\in\Th\},
\end{equation}
 we also have
\begin{equation}\label{inclu3fvk}
\div V^{\face}_{k-1}\equiv V^{\vol}_{k-1}.
\end{equation}
\begin{prop}
For the Virtual element spaces defined in \eqref{glo-n}, \eqref{glo-e}, \eqref{globf}, and \eqref{glo-v} the following is an exact sequence:
\begin{equation*}
\!\!\R
\vshortarrow{{\rm  i}}
V^{\node}_{k+1}(\Omega)
\vshortarrow{ \bf grad}
V^{\edge}_{k}(\Omega)
\vshortarrow{ \bcurl}
V^{\face}_{k-1}(\Omega)
\vshortarrow{\rm div}
V^{ \vol}_{k-1}(\Omega)
\vshortarrow{\rm o}
0 .
\end{equation*}
\end{prop}
\begin{remark}\label{catenak} Here too it is very important to point out that the inclusions  \eqref{inclu3nek}, \eqref{inclu3efk} and \eqref{inclu3fvk} are
(in a sense) also {\bf practical}, and not only theoretical. By this, more specifically, we mean that: given the degrees of freedom
of a $q\in V^{\node}_{k+1}$ we can compute the corresponding degrees of freedom of $\nabla q$ in $V^{\edge}_{k}$; and given the degrees of freedom of a $\vv\in V^{\edge}_{k}$ we can compute the corresponding degrees of freedom of $\bcurl\,\vv$ in $V^{\face}_{k-1}$; finally (and this is almost obvious) from the degrees of freedom of a $\ww\in V^{\face}_{k-1}$ we can compute its divergence in each element and obtain an element in $V^{\vol}_{k-1}$.\qed
\end{remark}

\subsection{Scalar products for VEM spaces in 3D}

\noindent From the local scalar products in $V^{\edge}_k(\PP)$  we can  also define a scalar product in $V^{\edge}_{k}$ in the obvious way
\begin{equation}\label{PSe3gk}
[\vv,\ww]_{e,\mu}:=\sum_{\PP\in\Th}[\vv,\ww]_{\epsmP}
\end{equation}
and we note that for some constants $\alpha_*$ and $\alpha^*$ independent of $h$
\begin{equation}\label{SP3boundsk}
\alpha_* \mu_0(\vv,\vv)_{0,\Om}\le [\vv,\vv]_{e,\mu}\le \alpha^* \mu_1(\vv,\vv)_{0,\Om}\qquad\forall \vv\in V^{\rm{e}}_{k} .
\end{equation}
It is also important to point out that, using \eqref{consiE3k} we have
\begin{equation}\label{consiE3gk}
[\vv,\pp]_{e,\mu}=(\mu\Pi^{0}_k\vv,\pp)_{0,\Om}\equiv\int_{\Om}\mu\Pi^{0}_k\vv\cdot\pp\dO \quad\forall \vv\in V^{\edge}_{k}, \,\forall\pp\mbox{ piecewise in $(\P_{k})^3$} .
\end{equation}
From \eqref{PSf3k} we can  also define a scalar product in $V^{\face}_{k-1}$ in the obvious way
\begin{equation}\label{PSf3gk}
[\vv,\ww]_{V^{\face}_{k-1}}:=\sum_{\PP\in\Th}[\vv,\ww]_{\fpsP}
\end{equation}
and we note that, for some constants $\alpha_1$ and $\alpha_2$ independent of $h$
\begin{equation}\label{SPf3boundsk}
\alpha_1 (\vv,\vv)_{0,\Om}\le [\vv,\vv]_{V^{\face}_{k-1}}\le \alpha_2 (\vv,\vv)_{0,\Om}\qquad\forall \vv\in V^{\face}_{k-1} .
\end{equation}
Note also that, using \eqref{consiF3k} we have
\begin{equation}\label{consif3k}
[\vv,\pp]_{{V^{\face}_{k-1}}}=(\vv,\pp)_{0,\Om}\equiv\int_{\Om}\vv\cdot\pp\dO \qquad\forall \vv\in V^{\face}_{k-1}, \,\forall\pp \mbox{ piecewise in $(\P_{k-1})^3$} .
\end{equation}


\section{The discrete problem and error estimates}\label{disc-pro}

\subsection{The discrete problem}

\noindent Given $\jj \in H_0(\div;\Om)$ with $\div\jj=0$, we construct its interpolant $\jj_I\in V^{\face}_{k-1}$ that matches all the degrees of freedom \eqref{dof-3dfk-1G}--\eqref{dof-3dfk-3G}:
\begin{align}
& \bullet \mbox{ $\forall f: \int_f
((\jj-\jj_I)\cdot\nn) p_{k-1} \df=0 \; \forall p_{k-1} \in \P_{k-1}(f), $} \label{dof-3dfk-1GI}\\
& \bullet\mbox{ $\forall \PP:\int_{\PP}
(\jj-\jj_I)\!\cdot\!\bgrad\, p_{k-1} \dPP =0\;\forall p_{k-1} \in \P_{k-1}(\PP), k> 1$} \label{dof-3dfk-2GI} \\
& \bullet\mbox{ $\forall \PP:\int_{\PP}
(\jj-\jj_I)\cdot (\xx_{\PP}\wedge {\bf p}_{k}) \dPP=0 \; \forall {\bf p}_{k} \in [\P_{k}(\PP)]^3$}.
 \label{dof-3dfk-3GI}
\end{align}

Then we can introduce the {\bf discretization} of
\eqref{K1_3}:
\begin{equation}\label{K1_3hk}
\left\{
\begin{aligned}
&\mbox{ find  }\HH_h\in V^{\edge}_{k} \mbox{ and }p_h\in V^{\node}_{k+1} \mbox{ such that: }\\
&[\bcurl\HH_h,\bcurl\vv]_{V^{\face}_{k-1}}+[\nabla p_h,\vv]_{e,\mu}~=~[\jj_I,\bcurl\vv]_{V^{\face}_{k-1}}
\quad\forall\vv\in V^{\edge}_{k}\\
&[\nabla q,\HH_h]_{e,\mu}~=~0~\quad\forall q\in V^{\node}_{k+1}.
\end{aligned}
\right.
\end{equation}
We point out that both $\bcurl\HH_h$ and $\bcurl\vv$ (as well as $\jj_I$) are {\it face Virtual Elements in} $V^{\face}_{k-1}(\PP)$ in each polyhedron $\PP$, so that
(taking also into account Remark \ref{catenak}) their {\it face} scalar products are computable as in \eqref{PSf3gk}.  Similarly, from the degrees of freedom of  a $q\in V^{\node}_{k+1}$ we can compute the degrees of freedom of
$\nabla q$, as an element of $V^{\edge}_{k}$, so that the two edge-scalar products in \eqref{K1_3hk} are computable as in \eqref{PSe3gk}.

\begin{prop}
\label{prop:uniqSol}
Problem \eqref{K1_3hk} has a unique solution $(\HH_h,p_h)$, and $p_h\equiv 0$.
\end{prop}
\begin{proof}
Taking $\vv=\nabla p_h$ (as we did for the continuous problem \eqref{K1_3}) in the first equation, and using \eqref{SP3boundsk}  we easily obtain $p_h\equiv0$ for \eqref{K1_3hk} as well. To prove uniqueness of $\HH_h$, set $\jj_I=0$, and let ${\overline \HH}_h$ be the solution of the homogeneous problem. From the first equation we deduce that $\curl\,{\overline \HH}_h=0$. Hence, from \eqref{rot03k} we have ${\overline \HH}_h= \nabla q^*_h$ for some $q^*_h \in V^n_{k+1}$. The second equation and \eqref{SP3boundsk} give then ${\overline \HH}_h=0$.
\end{proof}

In order to study the discretization error  between \eqref{K1_3} and 
\eqref  {K1_3hk} we need  the interpolant $\HH_I\in V^{\edge}_{k}$ of $\HH$, defined through the degrees of freedom \eqref{dof-3dek-1G}-\eqref{dof-3dek-5G}:
\begin{align}
& \bullet\mbox{  $\forall\; e$}: \mbox{$\int_e ((\HH-\HH_I)\cdot\tt_e) p_{k}\ds=0 \quad \forall p_{k} \in \P_{k}(e), $ } \label{intHH1}\\
& \bullet\mbox{ $\forall\;  f$}:  \mbox{$\int_f
\rot_f(\HH-\HH_I)^{\tau}\, p^0_{k-1} \df =0\quad \forall p^0_{k-1} \in \P^0_{k-1}(f)$ (for $k> 1$),} \label{intHH2}\\
& \bullet \mbox{ $ \forall\;  f$ with $\beta_f\ge 0$}: \mbox{$\int_f
((\HH-\HH_I)^{\tau}\cdot\xx_f) p_{\beta_f}\, \df =0\quad \forall p_{\beta_f} \in \P_{\beta_f}(f) ,$} \label{intHH3}\\
& \bullet \mbox{ $ \forall\;  \PP$}: \mbox{$\int_{\PP}
((\HH-\HH_I)\cdot\xx_{\PP}) p_{k-1}\, \dPP =0$}\quad \forall p_{k-1} \in \P_{k-1}(\PP) ,\label{intHH4} \\
& \bullet\mbox{ $\forall\; \PP$}:  \mbox{$\int_{\PP}
\bcurl(\HH-\HH_I)\cdot(\xx_{\PP}\wedge {\bf p}_{k})\,\dPP=0$} \quad \forall {\bf p}_{k} \in [\P_{k}(\PP)]^3.  \label{intHH5}
\end{align}
We have the following result.
\begin{prop} With the choices  \eqref{dof-3dfk-1GI}-\eqref{dof-3dfk-3GI}
and  \eqref{intHH1}-\eqref{intHH5}  we have 
\begin{equation}\label{hailmaryk}
\bcurl(\HH_I)=(\bcurl\HH)_{I}\equiv \jj_I.
\end{equation}
\end{prop}
\begin{proof} We should show that the {\it face degrees of freedom} \eqref{dof-3dfk-1G}-\eqref{dof-3dfk-3G} of the difference  $\bcurl\HH_I - \jj_I$
are  zero, that is:
\begin{align}
& \bullet\, \mbox{$\forall\, f: \int_f
((\bcurl\HH_I -\jj_I)\cdot\nn) {p_{k-1} \df =0\quad \forall p_{k-1}\! \in\! \P_{k-1}(f)}$}  \label{dof-3dfk-1Gp}\\
& \bullet\, \mbox{$\forall\,\PP:\,\int_{\PP}
(\bcurl\HH_I - \jj_I)\cdot {\bgrad p_{k-1} \dPP =0 \quad \forall p_{k-1}\! \in \!\P_{k-1}(\PP) }$}\label{dof-3dfk-2Gp} \\
& \bullet\,\mbox{$\forall\,\PP:\int_{\PP}
(\bcurl\HH_I - \jj_I)\cdot (\xx_{\PP}\wedge {\bf p}_{k}) \dPP =0\quad \forall {\bf p}_{k}\! \in \![\P_{k}(\PP)]^3$}
 \label{dof-3dfk-3Gp}
\end{align}
 From  the interpolation  formulas \eqref{dof-3dfk-1GI}-\eqref{dof-3dfk-3GI} we see that in \eqref{dof-3dfk-1Gp}-\eqref{dof-3dfk-3Gp} we can replace $\jj_I$ with $\jj$ (that in turn is equal to $\bcurl\HH$. Hence \eqref{dof-3dfk-1Gp}-\eqref{dof-3dfk-3Gp} become
\begin{align}
& \bullet\mbox{$\forall f: \int_f
\bcurl(\HH_I-\HH)\cdot\nn \,p_{k-1} \df =0\; \forall p_{k-1} \in \P_{k-1}(f), $} \label{ave1}\\
& \bullet\mbox{$ \forall \PP: \int_{\PP}\bcurl(\HH_I-\HH)\cdot {\bgrad p_{k-1} \dPP =0\;\forall p_{k-1} \in \P_{k-1}(\PP)},$} \label{ave2} \\
& \bullet\mbox{$\forall \PP: \int_{\PP}
\bcurl(\HH_I-\HH)\cdot (\xx_{\PP}\wedge {\bf p}_{k}) \dPP =0\; \forall {\bf p}_{k} \in [\P_{k}(\PP)]^3$}. \label{ave3}
\end{align}
Observing that \eqref{intHH1} and \eqref{intHH2} imply that
$$
\int_f
\rot_f(\HH-\HH_I)^{\tau}\, p_{k-1} \df =0\quad \forall p_{k-1} \in \P_{k-1}(f),
$$
and recalling that  on each $f$ the normal component of $\bcurl(\HH_I -\HH)$ is equal to the $\rot_f$ of the tangential components $(\HH_I-\HH)^{\tau}$, we deduce
\begin{equation*}
\int_f\bcurl(\HH_I -\HH)\cdot\nn p_{k-1} \df\equiv\int_f\rot_f(\HH_I-\HH)^{\tau}\,p_{k-1}\df =0.
\end{equation*}
Hence, \eqref{ave1} is satisfied. Next, we note that, having already \eqref{ave1} on each face, the equation \eqref{ave2} follows immediately with an integration by parts on $\PP$.
Finally, \eqref{ave3} is the same as \eqref{intHH5}, and the proof is concluded.
\end{proof}

We observe now that,  once we know that $p_h=0$, the first equation of \eqref{K1_3hk} reads
 \begin{equation}
 [\bcurl\HH_h,\bcurl\vv]_{V^{\face}_{k-1}}=
[\jj_I,\bcurl\vv]_{V^{\face}_{k-1}} \;\forall \vv\in V^{\edge}_{k},
 \end{equation}
 that in view of \eqref{hailmaryk} becomes
  \begin{equation}
 [\bcurl\HH_h-\bcurl\HH_I,\bcurl\vv]_{V^{\face}_{k-1}}=0
 \quad\forall \vv\in V^{\edge}_{k}.
 \end{equation}
 Using $\vv=\HH_h-\HH_I$ and \eqref{SPf3boundsk}, this easily implies
   \begin{equation}\label{bingo3k}
 \bcurl\HH_h=\bcurl\HH_I=\jj_I .
 \end{equation}

\subsection{Error estimates}\label{theo:est}

Let us bound the error $\HH-\HH_h$ in terms of approximation errors for $\HH$. From \eqref{bingo3k} we have
\begin{equation}\label{roteq3Dk}
\bcurl (\HH_I-\HH_h)=0 ,
\end{equation}
and therefore, from \eqref{diagramVn-Ve},
\begin{equation}\label{eungrad-a}
\HH_I-\HH_h= \nabla q^*_h \mbox{ for some }q^*_h\in V^{\node}_{k+1}.
\end{equation}
On the other hand, using \eqref{SP3boundsk} we have
\begin{equation}\label{stima-a1}
\alpha_*\mu_0\|\HH_I-\HH_h\|_{0,\Om}^2\le [\HH_I-\HH_h, \HH_I-\HH_h]_{e,\mu} .
\end{equation}
Then:
\begin{equation*}
\begin{aligned}
\alpha_*&\mu_0\|\HH_I-\HH_h\|_{0,\Om}^2\le[\HH_I-\HH_h, \HH_I-\HH_h]_{e,\mu}\\[2mm]
=&\mbox{(use \eqref{eungrad-a}) }[\HH_I-\HH_h, \nabla q^*_h]_{e,\mu}\\[2mm]
=&\mbox{(use the second of \eqref{K1_3hk}) }[\HH_I, \nabla q^*_h]_{e,\mu}\\[2mm]
=&\mbox{(add and subtract  $\Pi^{0}_k\HH$) }[\HH_I-\Pi^{0}_k\HH, \nabla q^*_h]_{e,\mu}+[\Pi^{0}_k\HH,\nabla q^*_h]_{e,\mu}\\[2mm]
=&\mbox{(use \eqref{consiE3gk}) }[\HH_I-\Pi^{0}_k\HH, \nabla q^*_h]_{e,\mu}+(\Pi^{0}_k\HH, \mu\Pi^{0}_k\nabla q^*_h)_{0,\Om}\\[2mm]
=&\mbox{(use the $2^\text{nd}$ of \eqref{K1_3}) }\underbrace{[\HH_I-\Pi^{0}_k\HH, \nabla q^*_h]_{e,\mu}}_{I}+\underbrace{(\Pi^{0}_k\HH, \mu\Pi^{0}_k\nabla q^*_h)_{0,\Om}
- (\HH, \mu\nabla q^*_h)_{0,\Om}}_{II}
\end{aligned}
\end{equation*}
For the first term we use \eqref{stima-edge} to get
\begin{equation}\label{est1}
I \le \mu_1\alpha^* \|\HH_I-\Pi^{0}_k\HH\|_{0,\Om} \|\nabla q^*_h\|_{0,\Om}.
\end{equation}
Next, following arguments similar to \cite{variable-primal} (Lemma 5.3), we obtain:
\begin{align}\label{est2}
II&=(\Pi^{0}_k\HH, \mu\Pi^{0}_k\nabla q^*_h)_{0,\Om}\!- \!(\HH, \mu\nabla q^*_h)_{0,\Om}\!+\!(\HH, \mu\Pi^{0}_k\nabla q^*_h)_{0,\Om}-(\HH, \mu\Pi^{0}_k\nabla q^*_h)_{0,\Om} \nonumber\\[2mm]
&=(\Pi^{0}_k\HH-\HH, \mu\Pi^{0}_k\nabla q^*_h)_{0,\Om}
+(\mu\HH, \Pi^{0}_k\nabla q^*_h-\nabla q^*_h)_{0,\Om}  \nonumber\\[2mm]
&=(\Pi^{0}_k\HH-\HH, \mu\Pi^{0}_k\nabla q^*_h)_{0,\Om}+(\mu\HH-\Pi^{0}_k\mu\HH, \Pi^{0}_k\nabla q^*_h-\nabla q^*_h)_{0,\Om} \\[2mm]
&\le \|\Pi^{0}_k\HH-\HH\|_{0,\Om}\|\mu\Pi^{0}_k\nabla q^*_h\|_{0,\Om}+\|\mu\HH-\Pi^{0}_k\mu\HH\|_{0,\Om}\|\Pi^{0}_k\nabla q^*_h-\nabla q^*_h\|_{0,\Om} \nonumber\\[2mm]
&\le \mu_1\|\Pi^{0}_k\HH-\HH\|_{0,\Om} \|\nabla q^*_h\|_{0,\Om}+\|\mu\HH-\Pi^{0}_k\mu\HH\|_{0,\Om}\|\nabla q^*_h\|_{0,\Om} . \nonumber
\end{align}
Inserting \eqref{est1}-\eqref{est2} in the above estimate we deduce
\begin{multline*}
\alpha_*\mu_0\|\HH_I-\HH_h\|_{0,\Om}^2\le \\
\Big(\mu_1  \alpha^* \|\HH_I-\Pi^{0}_k\HH\|_{0,\Om}
+\mu_1\|\Pi^{0}_k\HH-\HH\|_{0,\Om}+\|\mu\HH-\Pi^{0}_k\mu\HH\|_{0,\Om}\Big)\|\nabla q^*_h\|_{0,\Om}
\end{multline*}
that implies immediately (since $\alpha^*\ge 1$)
\begin{equation*}
\|\HH_I-\HH_h\|_{0,\Om}\le\frac{\mu_1\alpha^*}{\mu_0\alpha_*}\Big(  \|\HH_I-\Pi^{0}_k\HH\|_{0,\Om}
+\|\Pi^{0}_k\HH-\HH\|_{0,\Om}\Big) +\frac{1}{\mu_0\alpha_*}\|\mu\HH-\Pi^{0}_k\mu\HH\|_{0,\Om}.
\end{equation*}
Summarizing:
\begin{thm}\label{Th:estimate}
Problem \eqref{K1_3hk} has a unique solution, and we have
\begin{equation}\label{errH}
\|\HH-\HH_h\|_{0,\Om}\le C\, \Big(\|\HH-\HH_I\|_{0,\Om}+\|\Pi^{0}_k\HH-\HH\|_{0,\Om}+\|\mu\HH-\Pi^{0}_k\mu\HH\|_{0,\Om}\Big),
\end{equation}
with $C$ a constant depending on $\mu$ but independent of the mesh size. Moreover,
\begin{equation}\label{errj}
\|\bcurl (\HH -  \HH_h)\|_{0,\Om} = \|\jj-\jj_I\|_{0,\Om}.
\end{equation}
\end{thm}
The error bounds in \eqref{errH} and \eqref{errj} could then be expressed in terms of powers of $h$ (a suitable indicator of the mesh fineness) and of the regularity properties of $\HH$ and $\jj$, using classical approximation properties of VEM spaces, as described for instance
 in \cite{Acoustic, BLR-Stab, Steklov-VEM, Brenner, Brenner-2, Chen-Huang}. 
If the data $\mu$ and the solution $\HH$ are sufficiently regular, one obtains from \eqref{errH} that
\begin{equation}\label{piripacchio}
\|\HH -\HH_h\|_{0,\Om} \le C h^s \| \HH \|_{s,\Omega} \qquad 0 \le s \le {k+1} ,
\end{equation}
where the constant $C$ depend only on the polynomial degree $k$, the mesh regularity parameter $\gamma$ and on $\| \mu \|_{W^{k+1,\infty}(\Omega_h)}$.

\begin{remark}
By inspecting the proof of Theorem \ref{Th:estimate} we notice that, for this particular problem, the consistency property
\eqref{consif3k} for the space $V^{\face}_{k-1}$ is never used. Since only property \eqref{SPf3boundsk} is needed, in $V^{\face}_{k-1}$ we could simply take, for instance, as {\it scalar product in $V^{\face}_{k-1}$} the one (much cheaper to compute) associated to the norm
\begin{equation}\label{esc-bdf}
\|\vv\|_{V^{\face}_{k-1}}^2:=\sum_i (dof_i(\vv))^2 ,
\end{equation}
where  $dof_i$ are the degrees of freedom in $V^{\face}_{k-1}$ properly scaled.
\hfill \qed
\end{remark}

\section{Numerical Results}\label{sec:NE}

In this section we numerically validate the proposed VEM approach.
More precisely, we will focus on two main aspects of this method.
We will first show that we recover the theoretical convergence rate for standard and serendipity VEM,
then we compare these two approaches in terms of number of degrees of freedom.
For the present study we consider the cases $k=1$ and $k=2$.
A lowest order case (not belonging to the present family) has been already discussed  in~\cite{lowest-max3}.

In the following two tests we use four different types of decompositions of $[0,\,1]^3$:
\begin{itemize}
 \item \textbf{Cube}, a mesh composed by cubes;
 \item \textbf{Nine}, a regular mesh composed by 9-faced polyhedrons in accordance with a periodic pattern; 
 \item \textbf{CVT}, a Voronoi tessellation obtained by a standard Lloyd algorithm~\cite{cvtPaper};
 \item \textbf{Random}, a Voronoi tessellation associated with a set of seeds randomly 
 distributed inside $\Omega$.
\end{itemize}
Note that the meshes taken into account are of increasing complexity; 
in particular, the meshes \textbf{CVT} and \textbf{Random} have polyhedra with small faces and edges.

All discretizations have been generated with the c++ library \texttt{voro++}~\cite{voroPlusPlus}
and we exploit the software PARDISO~\cite{pardiso-6.0a,pardiso-6.0b} to solve the resulting linear systems.
\begin{figure}[!htb]
\begin{center}
\begin{tabular}{ccc}
\includegraphics[width=0.35\textwidth]{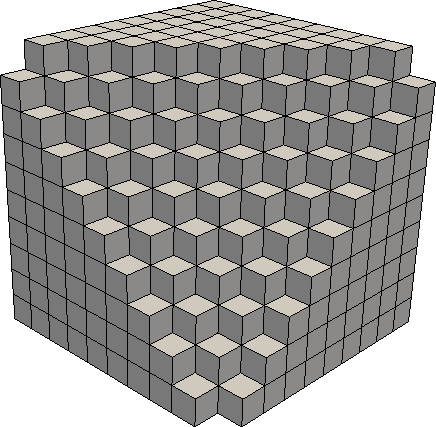} &\phantom{mm} & 
\includegraphics[width=0.35\textwidth]{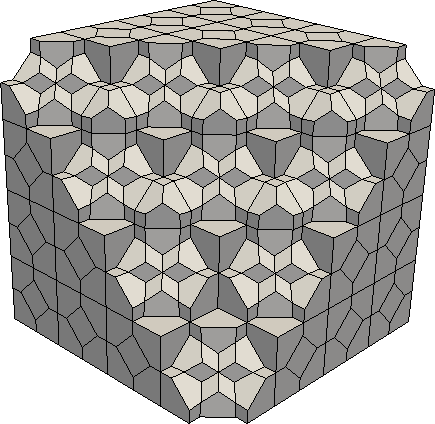}\\[1em]
Cube & &Nine \\[1em]
\includegraphics[width=0.35\textwidth]{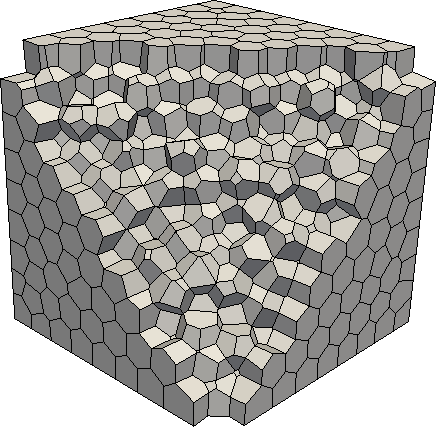} &\phantom{mm} &
\includegraphics[width=0.35\textwidth]{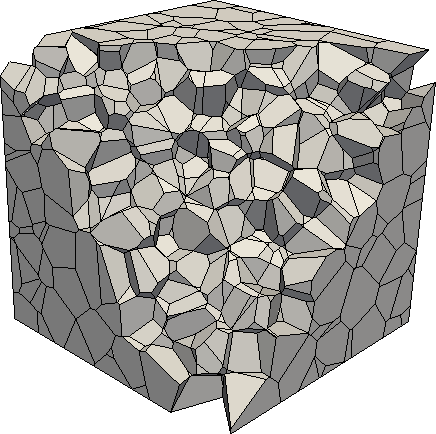}\\[1em]
CVT        & &Random 
\end{tabular}
\end{center}
\caption{A sample of the used meshes.}
\label{fig:cubes}
\end{figure}
In order to study the error convergence rate, for each type of mesh
we consider a sequence of three progressive refinements  composed by approximately 27, 125 and 1000 polyhedrons.
Then, we associate with each mesh a mesh-size 
$$
h:= \frac{1}{N_{\PP}} \sum_{i=1}^{N_{\PP}} h_{\PP}\,,
$$
where $N_{\PP}$ is the number of polyhedrons $\PP$ in the mesh and $h_{\PP}$ is the diameter of $\PP$.

Since $\HH_h$ is virtual, we use its projection $\Pi^0_k\HH_h$ to compute the $L^2$-error, i.e.,
the following norm is used as an indicator of the $L^2$-error:
$$
\frac{||\HH - \Pi^0_k\HH_h||_{0,\Omega}}{||\HH||_{0,\Omega}}\,.
$$
The expected convergence rate is $O(h^{k+1})$, see~\eqref{piripacchio}.

\vskip4truecm
{\color{white} .}

\noindent\textsl{\bf Test case 1: $h$-analysis}

We consider a problem 
with a constant permeability $\mu(\X)=1$. We take as exact solution 
\[
\HH(x,\,y,\,z) := \frac{1}{\pi}\left(\begin{array}{c}
                                     \sin(\pi y)- \sin(\pi z)\\
                                     \sin(\pi z)- \sin(\pi x)\\
                                     \sin(\pi x)- \sin(\pi y)
                                     \end{array}\right) ,
\]
and chose right-hand side and 
boundary conditions accordingly. 

In Figure~\ref{fig:convEse1} we show the convergence curves for each set of meshes.
The error behaves as expected ($O(h^2)$ and $O(h^3)$ for $k=1$ and $k=2$, respectively).

\begin{figure}[!htb]
\begin{center}
\begin{tabular}{cc}
\includegraphics[height=0.35\textwidth]{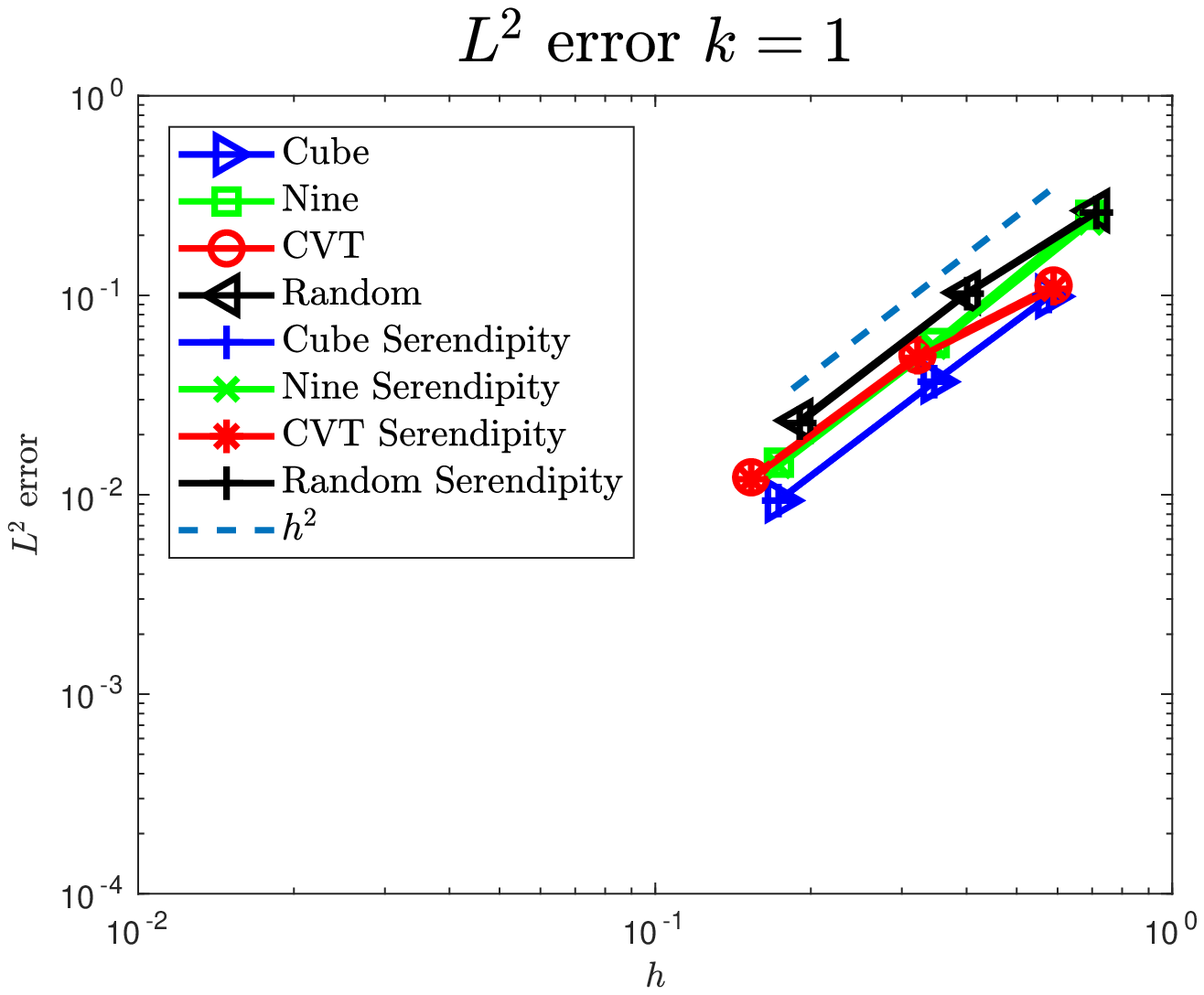} &
\includegraphics[height=0.35\textwidth]{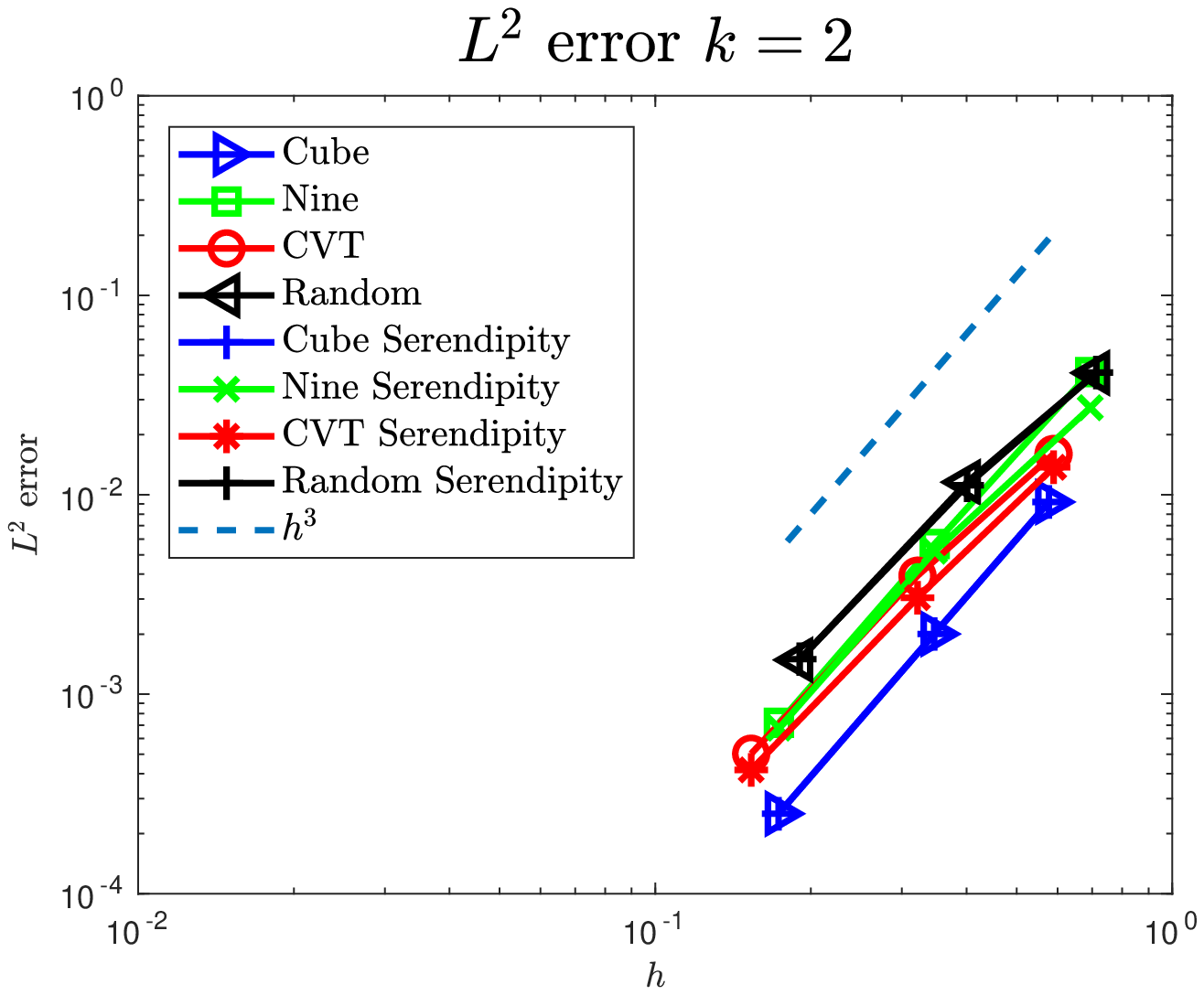}
\end{tabular}
\end{center}
\caption{Test case 1: convergence trend of the $L^2$-error for both standard and serendipity approach: 
case $k=1$ left, $k=2$ right.}
\label{fig:convEse1}
\end{figure}

From Figure~\ref{fig:convEse1} 
we also observe that 
we get almost the same values when we consider the standard or the serendipity approach.
These two methods are equivalent in terms of error, but the serendipity approach requires fewer degrees of freedom.
To better quantify the gain in terms of computational effort,
we compute the quantity 
$$
\texttt{gain} := \frac{\# dof_f -\# dof^S_f}{\# dof_f}\, 100\%\,,
$$
where $\# dof_f$ and $\# dof^S_f$ are the number of face degrees of freedom in  standard and  serendipity VEM.
We underline that in this computation we do not take into account the internal degrees of freedom 
since they can be removed via static condensation. 
As we can see from the data in Table~\ref{tab:gain}, the gain is remarkable (almost $50\%$ of the face d.o.f.s). Note that this also reflects on a much better performance of several solvers of the final linear system.

\begin{table}[!htb]
\begin{center}
\begin{tabular}{|c|c|c|c|c|c|c|c|c|}
\multicolumn{1}{c}{}&\multicolumn{8}{c}{\texttt{gain}}\\
\cline{2-9}
\multicolumn{1}{c|}{}&\multicolumn{4}{c|}{$k=1$}&\multicolumn{4}{c|}{$k=2$}\\[0.2em]
\hline
$\sim N_P$ &Cube &Nine &CVT &Random       &Cube &Nine &CVT &Random \\
\hline                                  
27   &56.6\% &51.0\% &50.2\% &50.3\% &56.4\% &52.0\% &49.9\% &50.4\% \\ 
125  &59.5\% &53.6\% &50.5\% &50.1\% &58.5\% &54.1\% &51.6\% &50.2\% \\ 
1000 &61.8\% &54.9\% &50.3\% &49.8\% &60.2\% &55.0\% &44.3\% &49.9\% \\ 
\hline
\end{tabular}
\end{center}
\caption{Test case 1: values of  {\rm{\texttt{gain}}} for each type of mesh taken into account.}
\label{tab:gain}
\end{table}

\newpage
\noindent\textsl{\bf Test case 2: $h$-analysis with a variable $\mu(\X)$}

We consider now a problem with variable 
permeability $\mu(\X)$ given by
$$
\mu(x,\,y,\,z) := 1+x+y+z .
$$
 We take as exact solution 
$$
\HH(x,\,y,\,z) := \frac{1}{(1+x+y+z)}\left(\begin{array}{c}
                                           \sin(\pi y)\\
                                           \sin(\pi z)\\
                                           \sin(\pi x)
                                           \end{array}\right) ,
$$
and we choose again right-hand side and 
boundary conditions accordingly.
In Figure~\ref{fig:convEse2} we provide the convergence curves for each set of meshes. The $L^2$-error behaves again as expected.

\begin{figure}[!htb]
\begin{center}
\begin{center}
\begin{tabular}{cc}
\includegraphics[height=0.35\textwidth]{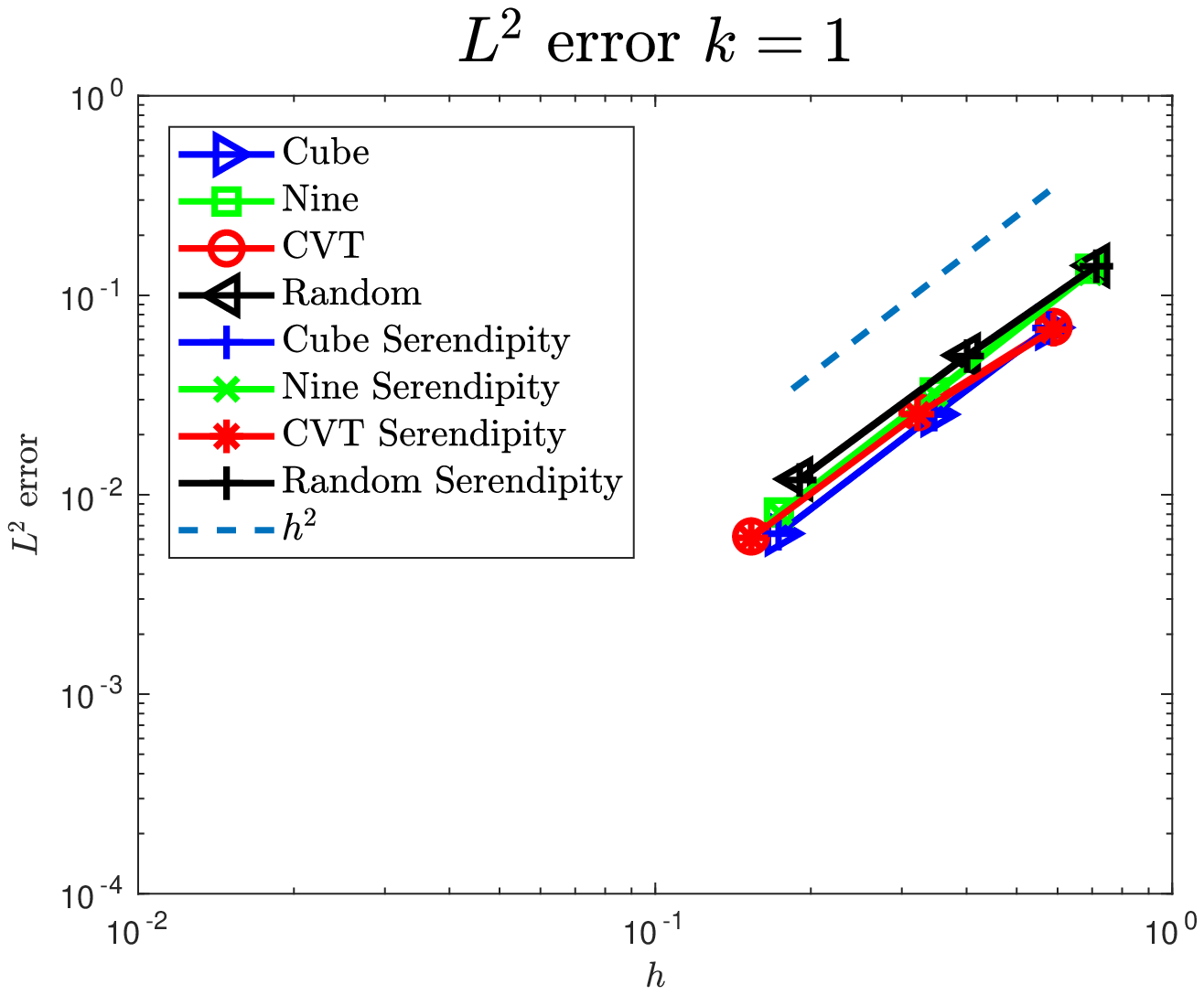} &
\includegraphics[height=0.35\textwidth]{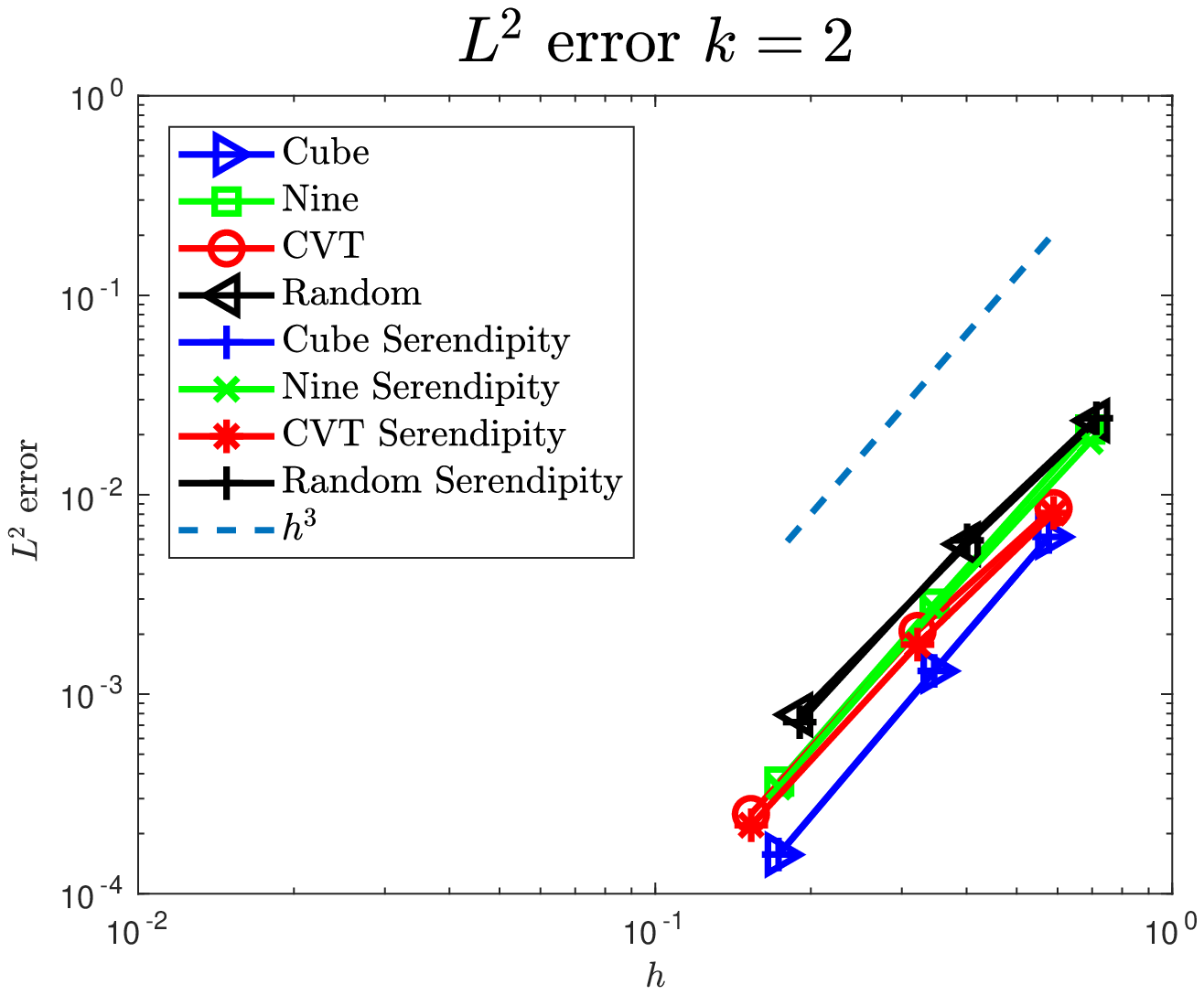}
\end{tabular}
\end{center}
\end{center}
\caption{Test case 2 - $L^2$-error for standard and serendipity approach: 
case $k=1$ and $k=2$.}
\label{fig:convEse2}
\end{figure}

\bibliographystyle{plain}

\bibliography{general-bibliography}

\end{document}